\documentclass[12pt]{amsart}
\usepackage{amsmath, amscd,stmaryrd}
\usepackage{amsthm}
\usepackage{verbatim}
\usepackage{amssymb}
\usepackage{mathtools}
\usepackage{xcolor}
\usepackage[all, cmtip]{xy}
\usepackage[pagebackref]{hyperref}

\newtheorem{lemma}{Lemma}[section]
\newtheorem{theorem}[lemma]{Theorem}
\newtheorem{corollary}[lemma]{Corollary}
\newtheorem{proposition}[lemma]{Proposition}

\theoremstyle{definition}
\newtheorem{definition}[lemma]{Definition}
\newtheorem{remark}[lemma]{Remark}

\newtheorem{example}[lemma]{Example}

\theoremstyle{remark}
\newtheorem*{remark*}{Remark}
\newtheorem*{note*}{Note}

\newenvironment{proof*}[1][\proofname]{%
  \proof[#1]}{\endproof}

\numberwithin{equation}{lemma}

\newcommand{\diag}{\operatorname{diag}}

\newcommand{\Bl}{\operatorname{Bl}}
\newcommand{\SymBl}{\operatorname{SymBl}}

\newcommand{\GL}{\operatorname{GL}}
\newcommand{\PGL}{\operatorname{PGL}}

\newcommand{\codim}{\operatorname{codim}}
\newcommand{\Spec}{\operatorname{Spec}}
\newcommand{\Proj}{\operatorname{Proj}}
\newcommand{\Sym}{\operatorname{Sym}}

\newcommand{\ord}{\operatorname{ord}}

\newcommand{\SL}{\operatorname{SL}}
\newcommand{\PSL}{\operatorname{PSL}}
\newcommand{\SO}{\operatorname{SO}}
\newcommand{\op}{\operatorname{op}}

\newcommand{\thickslash}{\mathbin{\!\!\pmb{\fatslash}}}
\newcommand{\Aff}{\mathbb{A}}
\newcommand{\Q}{\mathbb{Q}}

\newcommand{\ix}{\mathcal{X}}
\newcommand{\ig}{\mathcal{G}}
\newcommand{\iy}{\mathcal{Y}}
\newcommand{\ic}{\mathcal{C}}
\newcommand{\ie}{\mathcal{E}}
\newcommand{\iF}{\mathcal{F}}
\newcommand{\iu}{\mathcal{U}}

\newcommand{\cO}{{\mathcal{O}}}
\newcommand{\iz}{\mathcal{Z}}
\newcommand{\iA}{\mathcal{A}}
\newcommand{\iB}{\mathcal{B}}
\newcommand{\iC}{\mathcal{C}}

\newcommand{\iI}{\mathcal{I}}
\newcommand{\iJ}{\mathcal{J}}
\newcommand{\iK}{\mathcal{K}}
\newcommand{\iL}{\mathcal{L}}
\newcommand{\iM}{\mathcal{M}}
\newcommand{\iP}{\mathcal{P}}

\newcommand{\CC}{\mathbb{C}}
\newcommand{\G}{\mathbb{G}}

\newcommand{\ZZ}{\mathbb{Z}}
\newcommand{\GG}{\mathbb{G}}

\newcommand{\QQ}{\mathbb{Q}}
\newcommand{\PP}{\mathbb{P}}

\newcommand{\bX}{\mathbf{X}}
\newcommand{\bY}{\mathbf{Y}}
\newcommand{\bU}{\mathbf{U}}

\newcommand{\A}{\mathbb{A}}

\newcommand{\Pro}{\mathbb{P}}
\newcommand{\characteristic}{\operatorname{char}}

\newcommand{\sat}{\operatorname{sat}}
\renewcommand{\ss}{ss}

\newcounter{item-counter}
\newcommand{\Gmu}{\pmb{\mu}} % Group \mu
\newcommand{\Galpha}{\pmb{\alpha}} % Group \alpha
\newcommand{\gitq}{/\!\!/} % GIT quotient
\newcommand{\red}{\mathrm{red}} % red subscript (reduced scheme)
\newcommand{\reg}{\mathrm{reg}} % reg superscript (regular )
\newcommand{\rig}{\mathrm{rig}} % rig superscript (rigidification)
\newcommand{\tame}{\mathrm{tame}} % tame subscript (stable stack gerbe over tame stack)
\newcommand{\maxlocus}{\mathrm{max}} % max superscript (locus of points with max-dimensional stabilizers)
\newcommand{\itemref}[1]{\eqref{#1}} % references to enumerations
\newcommand{\myitemi}[1]{\renewcommand{\theenumi}{#1}\item}

\newcommand{\spref}[1]{\href{http://stacks.math.columbia.edu/tag/#1}{#1}}

\newcommand{\toto}[4]{{#1} \xymatrix{\ar@<.5ex>[r]^-{#3}\ar@<-.5ex>[r]_-{#4}&{#2}}}

\newdir{+}{{}*!/-9pt/{}}
\newdir{>+}{@{>}*!/-9pt/{}}
\newcommand{\equalizer}[2]{\xymatrix@1@M=0mm@C=10mm{#1%
\ar@<.5ex>@{+->+}[r] \ar@<-.5ex>@{+->+}[r] & #2}}
\date{Aug 9, 2020}
\begin{document}
\title[Canonical reduction of stabilizers of Artin stacks]{Canonical reduction of stabilizers for Artin stacks with good moduli spaces}

\author[D. Edidin]{Dan Edidin}
\address{Department of Mathematics, University of Missouri-Columbia, Columbia, Missouri 65211}
\email{edidind@missouri.edu}

\author[D. Rydh]{David Rydh}
\address{KTH Royal Institute of Technology\\Department of Mathematics\\SE\nobreakdash-100\ 44\ Stockholm\\Sweden}
\email{dary@math.kth.se}
%    General info

\thanks{The first author was partially supported by the Simons collaboration
  grant 315460 while preparing this article. The second author was partially
  supported by the Swedish Research Council, grants 2011-5599 and 2015-05554.}

\begin{abstract}
  We present a complete generalization of Kirwan's partial desingularization
  theorem on quotients of smooth varieties. Precisely, we prove that if $\ix$ is an irreducible Artin stack with stable good moduli space
$\ix \stackrel {\pi} \to \bX$, then there is a canonical sequence of
birational morphisms of Artin stacks $\ix_n \to \ix_{n-1} \to 
\ldots \to \ix_0 = \ix$ with the following properties: (1)
the maximum dimension of a stabilizer of a point of $\ix_{k+1}$ is strictly smaller than the
maximum dimension of a stabilizer of $\ix_k$ and the final stack
$\ix_n$ has constant stabilizer dimension; (2) the morphisms
$\ix_{k+1} \to \ix_k$ induce projective and birational morphisms of good
moduli spaces $\bX_{k+1} \to \bX_{k}$. If in addition
the stack $\ix$ is smooth, then each of the intermediate stacks
$\ix_k$ is smooth and the final stack $\ix_n$ is a gerbe over
a tame stack. In this case the algebraic space 
$\bX_n$ has tame quotient singularities and is a partial desingularization
of the good moduli space $\bX$. 

When $\ix$ is smooth our result can be combined with D.\ Bergh's recent destackification theorem for tame
stacks to obtain a full desingularization of the algebraic space $\bX$.
\end{abstract}

\maketitle
\setcounter{tocdepth}{1}
\vspace{-5mm}% HACK TO GET TOC ON ONE PAGE
\tableofcontents
%\setcounter{secnumdepth}{1}

%%%%%%%%%%%%%%%%%%%%%%%%%%%%%%%%%%%%%%%%%%%%%%%%%%%%%%%%%%%%%%%%%%%%%%%%%%%%%
\section{Introduction}
Consider the action of a reductive group $G$ on a smooth projective
variety $X$. For any ample $G$-linearized line bundle  on
$X$ there is a corresponding projective geometric invariant theory (GIT)
quotient $X\gitq G$.  If $X^{s} = X^{ss}$, then  $X\gitq G$
has finite quotient singularities. However, if $X^{s} \neq X^{ss}$, then
the singularities of $X\gitq G$ can be quite bad. In a classic
paper, Kirwan \cite{Kir:85} used a careful analysis of stable
and unstable points on blowups to prove that if $X^s \neq \emptyset$,
then there is a sequence of blowups along smooth centers $X_n \to
X_{n-1} \to \ldots \to X_0 = X$ with the following properties: (1) The final
blowup $X_n$ is a smooth projective $G$-variety with $X_n^{s} =
X_n^{ss}$. (2) The map of GIT quotients $X_n\gitq G \to X\gitq G$ is
projective and birational. Since $X_n\gitq G$ has only finite quotient
singularities, we may view it as a partial resolution of the very
singular quotient $X\gitq G$.

Kirwan's result can be expressed in the language of algebraic stacks by
noting that for linearly reductive groups, a GIT quotient $X\gitq G$
can be interpreted as the good moduli space of the quotient stack
$[X^{ss}/G]$.  The purpose of this paper is to give a complete
generalization of Kirwan's result to algebraic stacks.

Precisely, we
prove (Theorem \ref{thm.main}) that if $\ix$ is a (not necessarily smooth)
Artin stack with {\em stable} good
moduli space $\ix \stackrel{\pi} \to \bX$, then there is a canonical sequence of
birational morphisms of stacks $\ix_n \to \ix_{n-1} \to \ldots \to
\ix_0 = \ix$ with the following properties: (1) If $\ix$
is irreducible, then the maximum dimension
of a stabilizer of a point of $\ix_{k+1}$ is strictly smaller than the
maximum dimension of a stabilizer of $\ix_k$ and the final stack
$\ix_n$ has constant stabilizer dimension. (2) The morphisms
$\ix_{k+1} \to \ix_k$ induce projective and birational morphisms of good
moduli spaces $\bX_{k+1} \to \bX_{k}$.

When the stack $\ix$ is smooth, then each intermediate stack
$\ix_{k}$ is smooth. This follows because
$\ix_{k+1}$ is an open substack
of the blowup of $\ix_k$ along the closed smooth substack $\ix^{\maxlocus}$
parametrizing points with maximal dimensional stabilizer.
Since $\ix_n$ has constant stabilizer dimension, it follows
(Proposition \ref{prop.tamegerbe}) that its moduli space $\bX_n$ has
only tame quotient singularities. Thus our theorem gives a canonical
procedure to partially desingularize the good moduli space $\bX$.

Even in the special case of GIT quotients, our method allows us to
avoid the intricate arguments used by Kirwan. In addition, we are not
restricted to characteristic~$0$. However, Artin stacks with good moduli
spaces necessarily have linearly reductive stabilizers at closed points.
In positive
characteristic this imposes a strong condition on the stack. Indeed by
Nagata's theorem if $G$ is a linearly reductive group over a field of
characteristic $p$, then $G^0$ is diagonalizable and $p \nmid [G:G^0]$.

Theorem \ref{thm.main} can
be combined with the destackification results of
Bergh~\cite{Ber:17,BeRy:14} to give a functorial resolution of the singularities of
good moduli spaces of smooth Artin stacks in arbitrary characteristic
(Corollary~\ref{cor.res-good-quot-sing}).

In the smooth case, our results were applied in \cite{EdSa:17}
to study intersection theory
on singular good moduli spaces. There, Theorem \ref{thm.main}
is used to show that the pullback $A^*_{\op}(\bX)_\Q \to A^*(\ix)_\Q$
is injective, where $A^*_{\op}$ denotes the operational Chow ring
defined by Fulton \cite{Ful:84}. 

When $\ix$ is singular, but possesses a virtual smooth structure, a variant
of Theorem \ref{thm.main} can be applied to define numerical invariants of the stack. In 
the papers \cite{KiLi:13, KLS:17} the authors use a construction similar
to Theorem \ref{thm.main} for GIT quotients to define generalized Donaldson--Thomas invariants 
of Calabi--Yau three-folds. For GIT quotient stacks, 
the intrinsic blowup of \cite{KiLi:13, KLS:17} 
is closely related to the saturated blowup (Definition \ref{def.saturated-blowup})
of a stack $\ix$ along the locus of
maximal dimensional stabilizer 
(see Remark \ref{rem.intrinsic-blowups}). 
Similar ideas are also being considered in
recent work in progress of Joyce and Tanaka.

\subsection*{Outline of the proof of Theorem \ref{thm.main}}
The key technical construction in our proof is the {\em saturated} blowup
of a stack along a closed substack.
\subsubsection*{Saturated Proj and blowups.}
If $\ix$ is an Artin stack with good moduli space
morphism $\ix \stackrel{\pi} \to \bX$, then
the blowup of $\ix$ along a closed substack $\ic$ does not necessarily
have a good moduli space.
The reason is that if 
$\iA$ is any sheaf
of graded ${\mathcal O}_\ix$-modules, then $\Proj_\ix(\iA)$
need not have a good moduli space. However, we prove
(Proposition \ref{prop.saturatedproj}) that there is a canonical open substack
$\Proj^\pi_\ix(\iA) \subset \Proj_\ix(\iA)$ whose good moduli space
is $\Proj_\bX(\pi_* \iA)$.
In general the morphism $\Proj^\pi_\ix(\iA) \to \ix$ is not proper, but
the natural morphism
$\Proj_\bX(\pi_* \iA) \to \bX$ is identified with the morphism
of good moduli spaces induced from the morphism of stacks
$\Proj^\pi_\ix(\iA) \to \ix$. We call $\Proj^\pi_\ix(\iA)$ the {\em
saturated} $\Proj$ of $\iA$ relative to the good moduli space morphism $\pi$
(Definition \ref{def.saturated-proj}).

If $\ic$ is a closed substack of $\ix$, given by the sheaf of ideals $\iI$,
then we call $\Bl^\pi_\ic \ix := \Proj^\pi_\ix\bigl(\bigoplus \iI^n\bigr)$ the {\em saturated
blowup} of $\ix$ along $\ic$. A key fact is that when $\ix$ and $\ic$ are
smooth over a field, then
$\Bl^\pi_\ic\ix$ has a particularly simple description (Proposition \ref{prop.reichstein.is.saturatedproj}) as the {\em Reichstein transform} of $\ix$ along $\ic$ which we define in the next paragraph.

Given a closed substack $\ic \subset \ix$ the 
{\em Reichstein transform} $R(\ix,\ic)$ of $\ix$ along $\ic$ is
the open substack of the blowup $\Bl_\ic \ix$ whose complement is the strict transform
of the saturation of $\ic$ with respect to the good moduli space morphism $\ix \stackrel{\pi} \to \bX$.
The Reichstein
transform was introduced
in \cite{EdMo:12} where toric methods were used to prove that there is a canonical sequence
of toric Reichstein transforms, called stacky star subdivisions,
which turns a smooth  Artin toric stack
into a smooth Deligne--Mumford toric stack.
The term ``Reichstein transform'' was inspired by
Reichstein's paper \cite{Rei:89} which contains the result (loc. cit.  Theorem 2.4) that if
$C \subset X$ is a smooth, closed $G$-invariant subvariety of a smooth, $G$-projective variety
$X$ defined over a field
then $(\Bl_C X)^{ss}$ is the complement of the strict transform
of the saturation of $C \cap X^{ss}$ in the blowup of $X^{ss}$
along $C \cap X^{ss}$.

\subsubsection*{Outline of the proof when $\ix$ is smooth and connected over an algebraically closed field}
If $\ix$ is a smooth Artin stack over a field
with good moduli space $\ix \to \bX$, then the substack $\ix^\maxlocus$,
corresponding to points with maximal dimensional stabilizer, is closed
and smooth. Thus $\ix' = R(\ix, \ix^\maxlocus)$ is a smooth Artin stack
whose good moduli space $\bX'$ maps properly to $\bX$ and is an
isomorphism over the complement of $\bX^\maxlocus$, the image of
$\ix^\maxlocus$ in $\bX$. The stability hypothesis ensures that as long as
the stabilizers are not all of constant dimension, $\bX^\maxlocus$ is
a proper closed subspace of $\bX$. Using the local structure theorem
of \cite{AHR:15} we can show (Proposition \ref{prop.inductionstep})
that the maximum dimension of the
stabilizer of a point of $\ix'$ is strictly smaller than the maximum
dimension of the stabilizer of a point of $\ix$. The proof
then follows by induction.

In the local case, Proposition \ref{prop.inductionstep}
follows from Theorem \ref{thm.fixed} which states that if $G$ is a smooth, connected and linearly
reductive group acting on a smooth affine scheme, then the equivariant
Reichstein transform $R_G(X,X^G)$ has no $G$-fixed points. 
The proof
of Theorem \ref{thm.fixed} is in turn reduced to the case that $X = V$ is a representation of $G$, where the statement can be checked by
direct calculation (Proposition \ref{prop.repfixed}).

\subsubsection*{The general case over an algebraically closed field}
For a singular stack $\ix$, the strategy is essentially the same as in the smooth
case. The locus of points $\ix^\maxlocus$ of maximum dimensional stabilizer has a canonical substack structure but this need not be reduced. The proof is a bit technical, particularly in positive characteristic, and is given in Appendix~\ref{app.fixed}.
When
$\ix$ is singular, the Reichstein transform $R(\ix, \ix^\maxlocus)$ is not
so useful: the
maximum stabilizer dimension need not
drop (cf.\ Example \ref{ex.false}) and the Reichstein transform
need not admit a good moduli space (cf.\ Example \ref{ex.false2}).

The saturated blowup, however, always admit a good moduli space. Moreover,
we prove that if $\ix^\maxlocus \neq \ix$ and $\ix$ is irreducible,
then the saturated blowup $\Bl^\pi_{\ix^{\maxlocus}} \ix$ has strictly smaller dimensional
stabilizers (Proposition~\ref{prop.inductionstep-singular}).
This is again proved by reducing to the case that 
$\ix = [X/G]$ where $X$ is an affine scheme and $G$ is a linearly reductive group.
Since, $X$ can be embedded
into a representation $V$ of $G$, we can use the corresponding result for
smooth schemes and functorial properties of saturated blowups (Proposition 
\ref{prop.saturatedblowups-and-basechange}) to prove the result.

\subsubsection*{The general case}
When $\ix$ is not of finite type over a field the locus
$\ix^\maxlocus$ still has a natural substack structure but this requires a careful proof which we do in Appendix~\ref{app.fixed}. If $\ix$ is smooth over a
base $S$, then  $\ix^\maxlocus$ is also smooth over $S$.
The main theorem in this case then follows by
passing to fibers over $S$ (Section~\ref{ssec.smooth.over.base}). In general,
we embed $\ix$ into a stack which is smooth over a base scheme $S$. This is done in
Section~\ref{ssec.sing.over.base}.

\subsection*{Conventions and notation}
We use the term Artin stack and algebraic stack interchangeably.
All algebraic stacks are assumed to be noetherian with affine diagonal.
A morphism $\ix \to \iy$ of stacks in groupoids is a {\em gerbe}
if it satisfies the equivalent conditions of \cite[Tag \spref{06NY}]{stacks-project}. If $\ix$ and $\iy$ are algebraic stacks \cite[Tag \spref{06QB}]{stacks-project} implies that there is an fppf covering $U \to \iy$ with $U$ an algebraic space such that $\ix \times_\iy U = BG$ where $G$ is an fppf group space
over $U$. In this paper all stacks are algebraic and in particular all gerbes are algebraic. 

A {\em point} of an algebraic stack $\ix$ is an equivalence
class of morphisms $\Spec K \stackrel{x} \to \ix$ where $K$ is a field, and
$(x', K') \sim (x'', K'')$ if there is a $k$-field $K$ containing
$K', K''$ such that the morphisms $\Spec K \to \Spec K' \stackrel{x\smash{'}} \to \ix$
and $\Spec K \to \Spec K'' \stackrel{x\smash{''}} \to \ix$ are isomorphic.
The set of points of $\ix$ is denoted $|\ix|$.

Since $\ix$ is noetherian, every point $\xi \in | \ix|$ is
algebraic \cite[Th\'eor\`eme 
11.3]{LaMo:00}, \cite[Appendix B]{Rydh:11}. This means that if $\Spec K \stackrel{x} \to \ix$ is a representative for $\xi$, then 
the morphism $x$ factors as $\Spec K \stackrel{\smash{\overline{x}}\vphantom{x}}
\to \ig_\xi \to \ix$, where $\overline{x}$ is faithfully flat and $\ig_\xi \to \ix$
is a representable monomorphism. Moreover, $\ig_{\xi}$
is a gerbe over a field $k(\xi)$ which is called the \emph{residue field} of
the point $\xi$. The stack $\ig_{\xi}$ is called the \emph{residual gerbe}
and is independent of the choice
of representative  $\Spec K \stackrel{x} \to \ix$.

Given a morphism $\Spec K \stackrel{x} \to \ix$,
define the {\em stabilizer} $G_x$ of $x$ as the fiber product:
\vspace{-1mm}% HACK
\[
\xymatrix{
G_x \ar[d] \ar[r] &  \Spec K  \ar[d]^{(x,x)} \\
\ix \ar[r]^-{\Delta_\ix} & \ix \times_k \ix.
}
\]
Since the diagonal is affine, $G_x$ is an affine $K$-group scheme.

When we work over an algebraically closed field $k$, any closed point is geometric
and is represented by a morphism $\Spec k \stackrel{x} \to \ix$. In this case the residual gerbe is $BG_x$ where $G_x$ is the stabilizer of $x$.

\subsection*{Acknowledgments}
We would like to thank the referees for several useful and inspiring suggestions.

%%%%%%%%%%%%%%%%%%%%%%%%%%%%%%%%%%%%%%%%%%%%%%%%%%%%%%%%%%%%%%%%%%%%%%%%%%%%%

\section{Stable good moduli spaces and statement of the main theorem}

\subsection{Stable good moduli spaces}
\begin{definition}[{\cite[Definition 4.1]{Alp:13}}]
  A morphism $\pi \colon \ix \to \bX$ from an algebraic stack to an
  algebraic space is a {\em good moduli space} if
\begin{enumerate}
\item\label{item.gms.cohaff}
 $\pi$ is {\em cohomologically affine}, meaning that the pushforward functor $\pi_*$
on the category of quasi-coherent ${\mathcal O}_\ix$-modules is exact.

\item\label{item.gms.geom}
  The natural map ${\mathcal O}_\bX \to \pi_* {\mathcal O}_\ix$ is an isomorphism.
\end{enumerate}
More generally, a morphism of Artin stacks $\phi \colon \ix \to \iy$
satisfying conditions \itemref{item.gms.cohaff} and~\itemref{item.gms.geom} is called a {\em good moduli space morphism}.
\end{definition}
\begin{remark} The morphism $\pi$
  is universal for maps to algebraic spaces, so the algebraic space $\bX$
  is unique up to isomorphism~\cite[Theorem 6.6]{Alp:13}.
  Thus, we can refer to $\bX$ as \emph{the} good
  moduli space of $\ix$. The good moduli space is
  noetherian~\cite[Theorem 4.16(x)]{Alp:13} and the morphism $\pi$ is of finite
  type~\cite[Theorem A.1]{AHR:15}.
\end{remark}
\begin{remark} If $\ix \to \bX$ is a good moduli space,
then the stabilizer of any closed
point of $\ix$ is linearly reductive by \cite[Proposition 12.14]{Alp:13}.
\end{remark}

\begin{remark} \label{rem.goodvscoarse} Let $\ix$ be a stack with finite inertia $I \ix \to \ix$. 
By the Keel--Mori theorem, there is a coarse moduli space
$\pi \colon \ix \to \bX$. Following \cite{AOV:08} we say that $\ix$
is {\em tame} if $\pi$ is cohomologically affine. This happens precisely when
the stabilizer groups are linearly reductive.
In this case
$\bX$ is also the good moduli space of $\ix$ by \cite[Example 8.1]{Alp:13}.
Conversely, if $\pi \colon \ix \to \bX$ is the good moduli space
of a stack such that all the stabilizers are 0-dimensional, then $\ix$
is a tame stack with coarse moduli space $\bX$ (Proposition \ref{prop.tameisgood}). More generally, if the stabilizers of $\ix$ have constant dimension
$n$ and $\ix$ is reduced, then $\ix$ is a gerbe over a tame stack whose coarse space is $\bX$
(Proposition \ref{prop.tamegerbe}).
\end{remark}

\begin{definition} \label{def.stablegms} Let $\pi \colon \ix \to
  \iy$ be a good moduli space morphism. A point $x$ of $\ix$ is {\em
   stable} relative to $\pi$ if $\pi^{-1}(\pi(x)) = \{x\}$ under the induced map of
  topological spaces $|\ix| \to |\iy|$. A point $x$ of $\ix$ is {\em
    properly stable} relative to $\pi$ if it is stable and $\dim G_x = \dim G_{\pi(x)}$.

We say $\pi$ is a stable (resp. properly stable) good moduli space morphism
if the set of stable (resp. properly stable) points is dense.
\end{definition}

The dimension of the fibers of the relative inertia
morphism $I \pi \to \ix$ is an upper semi-continuous function~\cite[Expos\'e VIb, Proposition 4.1]{SGA3}. Hence
the set $\ix_{> d} = \{x \in |\ix| : \dim G_x - \dim G_{\pi(x)} > d\}$ is closed.

\begin{proposition} \label{prop.stable}
The set of stable points defines an open (but possibly empty)
 substack $\ix^s \subset \ix$ which
is saturated with respect to the morphism $\pi$. If $\ix$ is irreducible
with generic point $\xi$, then
\[
\ix^s = \ix\smallsetminus \pi^{-1}(\pi(\ix_{>d}))
\]
where $d=\dim G_\xi - \dim G_{\pi(\xi)}$ is the minimum dimension of the
relative stabilizer groups.
In particular, $\dim G_x - \dim G_{\pi(x)}=d$ at all points of $\ix^s$.
\end{proposition}

\begin{proof}
If $Y \to \iy$ is any morphism, then
$(Y \times_{\iy} \ix)^s = Y \times_{\iy} \ix^s$. Indeed, if $x$ is stable,
then $\ig_{x}\to \ig_{\pi(x)}$ is a good moduli space \cite[Lemma
  4.14]{Alp:13}, hence a gerbe, so $\pi^{-1}(\pi(x)) \to \pi(x)$ is universally
injective. We can thus reduce to the case where $\iy = \bX$ is a scheme.

If $\iz$ is an irreducible component of $\ix$, then the map $\iz \stackrel{\pi|_{\iz}}
\to \pi(\iz)$ is a good moduli space morphism by \cite[Lemma 4.14]{Alp:13}.

If $x$ is a point of $\ix$, then $\pi^{-1}(\pi(x)) = \bigcup_{\iz \subset \ix} (\pi|_{\iz})^{-1}(\pi(x))$
where the union is over the irreducible components of $\ix$ which contain $x$.
Thus a point $x$ is stable if and only if $(\pi|_\iz)^{-1}(\pi|_{\iz}(x)) = \{x\}$
for every irreducible component $\iz$ containing $x$. If we let $\iz^s$ be the
set of stable points for the good moduli space morphism $\pi|_{\iz}$, then
$\ix^s = \left(\bigcup_{\iz} (\iz \smallsetminus \iz^s)\right)^c$ where
the union is over all irreducible components of $\ix$. Since we assume that $\ix$ is noetherian there are only a finite number of irreducible components. Thus, it suffices to prove that $\iz^s$ is open for each irreducible component $\iz$. In other words, we are reduced to the case that $\ix$ is irreducible.

Then $\ix_{\leq d}=\ix_d=\{x \in |\ix| : \dim G_x = d \}$ is open and dense
and to see that
$\ix^s = \left(\pi^{-1}(\pi(\ix_{>d}))\right)^c$ we argue as follows.

By \cite[Proposition 9.1]{Alp:13} if $x$ is a point of $\ix$ and
$\pi^{-1}(\pi(x))$ is not a singleton, then $\pi^{-1}(\pi(x))$
contains a unique closed point
$y$ and $\dim G_y$ is greater than the dimension of any other stabilizer in
$\pi^{-1}(\pi(x))$. The point $y$ is clearly not in the open set $\ix_d$, so we conclude
that $(\ix^s)^c \subset \pi^{-1}(\pi(\ix_{>d}))$ or equivalently that $\ix^s \supset \left(\pi^{-1}(\pi(\ix_{>d}))\right)^c$.

To obtain the reverse inclusion we need to show that if $x$ is a point of $\ix$
and $\pi^{-1}(\pi(x)) = \{x\}$, then $\dim G_x = d$. Consider the stack
$\pi^{-1}(\pi(x))$ with its reduced stack structure. The monomorphism
from the residual gerbe
$\ig_x \to \ix$ factors through a monomorphism $\ig_x \to \pi^{-1}(\pi(x)))$. Since
$\pi^{-1}(\pi(x))$ has a single point, the morphism $\ig_x \to \pi^{-1}(\pi(x))_\red$ is an equivalence~\cite[Tag \spref{06MT}]{stacks-project}. Hence $\dim \pi^{-1}(\pi(x)) = \dim \ig_x = -\dim G_x$.

Let $\xi$ be the unique closed point in the generic fiber of $\pi$. Then
$x\in \overline{\{\xi\}}$ so by upper semi-continuity
$\dim G_x\geq \dim G_\xi$ and $\dim \pi^{-1}(\pi(x)) \geq \dim \pi^{-1}(\pi(\xi))$. Moreover, $\dim \pi^{-1}(\pi(\xi)) \geq -\dim G_\xi$ with equality if and
only if
$\pi^{-1}(\pi(\xi))$ is a singleton. It follows that
\vspace{-2mm}% HACK
\[
\dim \pi^{-1}(\pi(\xi)) \geq -\dim G_\xi \geq -\dim G_x = \dim \pi^{-1}(\pi(x))
\]
is an equality so the generic fiber $\pi^{-1}(\pi(\xi))$ is a singleton and $\dim G_x=\dim G_\xi=d$.
\end{proof}

Let $\ix$ be a reduced and irreducible Artin stack and let 
$\pi \colon \ix \to \bX$ be a good moduli space morphism  with $\bX$ an algebraic space and let $\bX^s = \pi(\ix^s)$. 
Since $\ix^s$ is saturated, $\bX^s$ is open in $\bX$.

\begin{proposition} \label{prop.gerbe}
With notation as in the preceding paragraph
$\ix^s$ is a gerbe over a tame stack with coarse
space $\bX^s$. Moreover, $\ix^s$ is the largest saturated open substack with this property.
\end{proposition}
\begin{proof}
  By Proposition \ref{prop.stable} the dimension of the stabilizer
  $G_x$ is constant at every point $x$ of $\ix^s$.
  Hence by Proposition \ref{prop.tamegerbe}, $\ix^s$ is a gerbe
  over a tame stack whose coarse space is $\bX^s$.

  Conversely, if $\iu$ is a saturated open substack which is a gerbe over a
  tame stack $\iu_\tame$, then the good moduli space morphism $\iu\to \bU$
  factors via $\iu_\tame$. Since $|\iu|\to |\iu_\tame|$ and $|\iu_\tame|\to
  |\bU|$ are homeomorphisms, it follows that $\iu\subset \ix^s$ by definition.
\end{proof}

\subsection{Examples}
In the following examples, we work over some field $k$.

\begin{remark}
  If $\ix = [X/G]$ is a quotient stack with $X = \Spec A$ an affine
  variety and $G$ is a linearly reductive group, then the good moduli
  space morphism $\ix \to \bX=\Spec A^G$ is stable if and only if the
  action is stable in the sense of \cite{Vin:00}. This means that
  there is a closed orbit of maximal dimension. The morphism $\ix
  \to \bX$ is properly stable if the maximal dimension equals $\dim
  G$. Following \cite{Vin:00} we will say that a representation $V$ of a
  linearly reductive group $G$ is {\em stable} if the action of $G$ on
  $V$ is stable.
\end{remark}

\begin{example}
  If $X = \A^2$ and $G=\GG_m$ acts on $X$ by $\lambda(a,b) = (\lambda a,
  b)$, then the good moduli space morphism $[X/G] \to \A^1$ is not
  stable since the inverse image of every point under the quotient map
  $\A^2 \to \A^1$, $(a,b) \mapsto b$ contains a point with stabilizer
  of dimension $1$. On the other hand, if we consider the action of
  $\GG_m$ given by $\lambda(a,b) = (\lambda^d a, \lambda^{-e} b)$ 
with $d,e > 0$, then
  the good moduli space morphism $[X/G] \to \A^1$ is properly stable,
  since the inverse image of the open set $\A^1 \smallsetminus \{0\}$
  is the Deligne--Mumford substack $[\left(\A^2 \smallsetminus
    V(xy)\right)/\GG_m]$.
\end{example}

\begin{example} Consider the action of $\GL_n$ on $\mathfrak{gl}_n$ via
conjugation in characteristic zero.  If we identify $\mathfrak{gl}_n$ with the space
$\A^{n^2}$ of $n \times n$ matrices, then the map 
$\mathfrak{gl}_n \to \A^n$ 
which sends a matrix to the coefficients of its characteristic polynomial
is a good quotient, so the map $\pi \colon [\mathfrak{gl}_n/\GL_n] \to \A^n$ is a good moduli space morphism.
The orbit of  an $n \times n$ matrix is closed if and only if it is diagonalizable. Since the stabilizer of a matrix with distinct eigenvalues is
a maximal torus $T$, such matrices have orbits
of 
dimension $n^2 - n = \dim \GL_n - \dim T$ which is maximal.

If $U \subset \A^n$ is the open set corresponding to polynomials with distinct
roots, then $\pi^{-1}(U)$ is a $T$-gerbe over the scheme $U$. Hence $\pi$ is
a stable good moduli space morphism, although it is not properly stable.
\end{example}

\subsection{Statement of the main theorem}
\begin{theorem} \label{thm.main}
  Let $\ix$ be an Artin stack, of finite type over a noetherian
  algebraic space $S$,
  with stable good moduli space $\ix \stackrel{\pi} \to \bX$ and
  let $\ie\subset \ix$ be an effective Cartier divisor (possibly empty).
There is a canonical
sequence of saturated blowups of Artin stacks
$\ix_n \to \ix_{n-1} \to \ldots \to \ix_{0} = \ix$,
along closed substacks $(\ic_\ell\subset \ix_\ell)_{0\leq \ell\leq n-1}$, and effective Cartier divisors
$(\ie_\ell\subset \ix_\ell)_{0\leq \ell\leq n}$, $\ie_0=\ie$,
with the following properties
for each $\ell=0,1,\dots,n-1$.
\begin{enumerate}
\myitemi{1a}\label{item.mainthm.center}
  $|\ic_\ell|$ is the locus
  in $\ix_\ell\smallsetminus \ix_\ell^s$ of points of maximal dimensional stabilizer.
\myitemi{1b}\label{item.mainthm.exc}
  $\ie_{\ell+1}$ is the inverse image of $\ic_\ell\cup \ie_\ell$.
\myitemi{1c}\label{item.mainthm.smooth}
  If $\ix$ is smooth over $S$ and $\ie_0$ is snc relative to $S$ (Definition~\ref{def.snc}), then
  $\ic_{\ell}$ and $\ix_{\ell+1}$ are smooth over $S$ and
  $\ie_{\ell+1}$ is snc relative to $S$.
\myitemi{2a}\label{item.mainthm.stablegms}
  There is a stable good moduli space
$\pi_{\ell+1}\colon \ix_{\ell+1} \to \bX_{\ell+1}$. If $\pi$ is properly stable, then so is
$\pi_{\ell+1}$.
\myitemi{2b}\label{item.mainthm.birational}
  The morphism $f_{\ell+1}\colon \ix_{\ell+1} \to \ix_{\ell}$ induces an isomorphism
  $\ix_{\ell+1}\smallsetminus f_{\ell+1}^{-1}(\ic_\ell)\to \ix_\ell\smallsetminus
  \pi_\ell^{-1}\bigl(\pi_\ell(\ic_\ell)\bigr)$. In particular, we have an isomorphism
  $\ix^s_{\ell+1}\smallsetminus \ie_{\ell+1}\to \ix^s_\ell\smallsetminus \ie_\ell$.
\myitemi{2c}\label{item.mainthm.gmsprojbirational}
  The morphism $\ix_{\ell+1} \to \ix_{\ell}$ induces a 
blowup of good moduli spaces $\bX_{\ell+1} \to \bX_{\ell}$ which is an
isomorphism over $\bX_{\ell} \smallsetminus \pi_{\ell}(\ic_{\ell})$
and in particular over $\bX_\ell^s$.
\myitemi{3}\label{item.mainthm.stabdim}
  The maximum dimension of the stabilizers of points of
  $\ix_{\ell+1} \smallsetminus \ix_{\ell+1}^s$ is strictly smaller than the maximum dimension of the stabilizers of points of 
$\ix_{\ell} \smallsetminus \ix^s_{\ell}$.
\end{enumerate}
The final stack $\ix_n$ has the following properties:
\begin{enumerate}\setcounter{enumi}{7}
\myitemi{4a}\label{item.mainthm.final.stable}
  Every point of $\ix_n$ is stable. In particular, $\pi_n$ is a
  homeomorphism and the dimension of the stabilizers of $\ix_n$ is locally
  constant.
\myitemi{4b}\label{item.mainthm.final.exc}
  $\ix_n\to \ix$ is an isomorphism over $\ix^s$ and
  $\ix_n\smallsetminus \ie_n=\ix^s\smallsetminus \ie$. In particular, $\ix^s\subset \ix_n$ is
  schematically dense.
\myitemi{4c}\label{item.mainthm.final.propstab}
  If $\ix$ is properly stable, then $\ix_n$ is a tame stack and
  $\bX_n$ its coarse moduli space.
\myitemi{4d}\label{item.mainthm.final.gerbe}
  If $\ix^s\smallsetminus \ie$ is a gerbe over a tame stack (e.g., if $\ix^s$ is reduced), then
  $\ix_n$ is a gerbe over a tame stack.
\end{enumerate}
The tame stack above is separated if and only if $\bX$ is separated.
The sequence $\ix_n\to\ix_{n-1}\to\ldots\to \ix$ does not depend on $\ie$.
\end{theorem}

\begin{remark}
  If $\ix$ is smooth over a field, then the morphisms 
  $\ix_{\ell+1} \to \ix_{\ell}$ are
  all \emph{Reichstein transforms} with centers $\ic_\ell$. The closed substack
  $\ic_\ell$ is the set of points in $\ix_{\ell} \smallsetminus \ix_\ell^{s}$
  with maximal-dimensional stabilizer equipped with a canonical scheme
  structure which need not be reduced if $\ix$ is singular.
  In particular, $\ic_\ell\cap
  \ix_\ell^s=\emptyset$ and $\ix_n\to \ix$ is an isomorphism over the
  stable locus. Note that if $\ix_\ell$ is irreducible, then no point with maximal
  dimensional stabilizer can be stable, so $\ic_\ell$ is supported on the locus
  of points with maximal dimensional stabilizer in $\ix_\ell$.
\end{remark}

\begin{remark}
  For singular $\ix$, there are other possible sequences
  $\ix_n\to\ix_{n-1}\to\ldots\to \ix$ that satisfy the conclusions: (1a),
  (2a)--(2c), (3), (4a) and (4c) but for which $f_{\ell+1}^{-1}(C_\ell)$ is not a
  Cartier divisor. In our sequence, using saturated blowups, $\ix^s\subset
  \ix_n$ is schematically dense and there are no new irreducible components
  appearing in the process. It is, however, possible to replace the saturated
  blowups with variants such as saturated symmetric blowups or intrinsic
  blowups, see Remark~\ref{rem.intrinsic-blowups}. For these variants, $\ix^s$
  is typically not schematically dense and $\ix_n$ can have additional irreducible
  components.
\end{remark}
\begin{definition}[{see \cite[Definition 3.1]{BeRy:14}}]\label{def.snc}
Let $\ix\to S$ be a smooth morphism. An effective Cartier divisor $\ie$ on
$\ix$ has \emph{simple normal crossings} or is \emph{snc} if $\ie=\bigcup
\ie_i$ where the $\ie_i$ are effective Cartier divisors on $\ix$ such that
for every $J\subset I$ it holds that
$\ie_J:=\bigcap_{i\in J} \ie_i$ has codimension $|J|$ in $\ix$ and is smooth
over $S$. A closed substack $\ic\subset \ix$ has \emph{normal crossings}
with $\ie$ if $\ic\cap \ie_J$ is smooth over $S$ and $\codim_x (\ic\cap \ie_J,\ic)=\sum_{i\in J} \codim_x(\ic\cap \ie_j,\ic)$ for every
$J\subset I$ and $x\in |\ic\cap \ie_J|$.
\end{definition}

%%%%%%%%%%%%%%%%%%%%%%%%%%%%%%%%%%%%%%%%%%%%%%%%%%%%%%%%%%%%%%%%%%%%%%%%%%%%%
\section{Saturated Proj and saturated blowups}
In this section we study a variant of Proj and blowups that depends on
a morphism $\pi$. In every result, $\pi$ will be a good moduli space
morphism but we define the constructions for more general morphisms.

\subsection{Saturated Proj, saturated blowups and their good moduli spaces}
\begin{definition}[Saturated $\Proj$] \label{def.saturated-proj}
Let $\pi \colon \ix \to \iy$ be a morphism of algebraic stacks
and let $\iA$ be a (positively) graded sheaf of finitely generated
${\mathcal O}_{\ix}$-algebras. Let $\pi^{-1}\pi_*\iA_+$ denote the image
of the natural homomorphism $\pi^*\pi_*\iA_+\to \iA_+\to \iA$.
Define $\Proj^\pi_\ix \iA = 
\Proj_\ix \iA \smallsetminus
V(\pi^{-1}\pi_*\iA_+)$. We call $\Proj^\pi_\ix \iA$ the {\em saturated $\Proj$}
of $\iA$ relative to the morphism $\pi$.
\end{definition}
Note that the morphism $\Proj^\pi_\ix \iA \to \ix$ need not be proper.  Also
note that there is a canonical morphism $\pi' \colon \Proj^\pi_\ix\iA\to \Proj_\iy
\pi_*\iA$: we are exactly removing the locus where $\Proj_\ix\iA\dashrightarrow \Proj_\iy \pi_*\iA$ is not defined.
We will show that if $\pi$ is a good moduli space morphism, then $\pi'$ is a good moduli space morphism (Proposition~\ref{prop.saturatedproj}).

\begin{definition}[Saturated blowups] \label{def.saturated-blowup}
Let $\pi \colon \ix \to \iy$ be a morphism of algebraic stacks and let $\ic
\subset \ix$ be a closed substack with sheaf of ideals $\iI$. We let
$\Bl^\pi_\ic \ix=\Proj^\pi_\ix \bigl(\bigoplus \iI^n\bigr)$ and call it the
\emph{saturated blowup} of $\ix$ in $\ic$. The \emph{exceptional divisor}
of the saturated blowup is the restriction of the exceptional divisor
of the blowup along the open substack $\Bl^\pi_\ic\ix \subset \Bl_\ic \ix$.
\end{definition}

We will later see that the good moduli space $\iy':=\Proj_\iy \bigl(\bigoplus
\pi_* \iI^n\bigr)$ of $\ix':=\Bl^\pi_\ic \ix$ is the blowup of $\iy$ in
$\pi_*(\iI^n)$ for all sufficiently divisible $n$ (Proposition
\ref{prop.saturatedblowups} \itemref{item.saturatedblowup.gms-blowup}).

The following example shows that the saturated Proj should be viewed as a
stack-theoretic generalization of GIT quotients. This is further developed in
Section \ref{sec.saturated-proj-GIT}.

\begin{example}[Saturated Proj and GIT] \label{ex.saturated-proj}
  Let $X$ be a projective variety with the action of
  a linearly reductive group $G$. Let $L$ be a $G$-linearized ample line bundle.
The section ring $A =\Gamma_*(X,L)$
  is a $G$-module and corresponds to a sheaf of graded ${\mathcal O}_{BG}$-algebras ${\mathcal A}$
  on the classifying stack $BG$. The structure map $\pi \colon BG \to \Spec k$
  is a good moduli space. The stack $[X/G] = \Proj_{BG} {\mathcal A}$
  typically has no good moduli space but the open substack
  $[X^{ss}(L)/G] = \Proj^\pi_{BG} {\mathcal A}$ has good moduli space
  $X\gitq G = \Proj \pi_*{\mathcal A} = \Proj A^G$.
\end{example}
\begin{proposition} \label{prop.saturatedproj}
If $\pi\colon \ix\to \iy$ is a good moduli space morphism and
$\iA$ is a finitely generated graded $\cO_{\ix}$-algebra, then
$\Proj^\pi_\ix \iA \to \Proj_\iy \pi_* \iA$ is a good moduli space morphism
and the morphism of good moduli spaces induced by the morphism
$\Proj^\pi_\ix \iA \to \ix$ is the natural morphism
$\Proj_Y \pi_* \iA \to \iy$.
\end{proposition}

\begin{proof}
  To show that the natural morphism
  \[
  \Proj^\pi_\ix \iA \to \Proj_\iy \pi_* \iA
  \]
  is a good moduli
  space morphism, we may, by \cite[Proposition 4.9(ii)]{Alp:13}, work locally
  in the smooth or fppf topology on $\iy$ and assume that $\iy$ is affine. In
  this case $\Proj \pi_*\iA$ is the scheme obtained by gluing the affine
  schemes $\Spec (\pi_*\iA)_{(f)}$ as $f$ runs through elements $f \in
  \pi_*\iA_+$.  Likewise, $\Proj^\pi_\ix\iA$ is the open set in $\Proj_\ix \iA$
  obtained by gluing the $\ix$-affine stacks $\Spec_\ix \iA_{(f)}$ as $f$ runs
  through $\pi_*A_+$. It is thus enough to prove that
  \[
  \Spec_\ix \iA_{(f)}\to \Spec_\iy (\pi_*\iA)_{(f)}
  \]
  is a good moduli space morphism.

By \cite[Lemma 4.14]{Alp:13} if $\iA$ is a sheaf of coherent
${\mathcal O}_\ix$-algebras, then $\Spec_\ix  \iA \to \Spec_\iy \pi_* \iA$ is a good moduli space morphism and the diagram 
\[
\xymatrix{%
\Spec_\ix\iA\ar[r]\ar[d] & \ix\ar[d]^{\pi} \\
\Spec_\iy\pi_*\iA\ar[r] & \iy
}
\]
is 
commutative.
Since good moduli space morphisms are invariant under
base  change \cite[Proposition 4.9(i)]{Alp:13} we see that
\[
\Spec_\ix \iA \smallsetminus
V(\pi^{-1}(\pi_*\iA_+)) \to \Spec_\iy \pi_*\iA \smallsetminus
V(\pi_*\iA_+)
\]
is a good moduli space morphism. Now $\Proj_\iy \pi_*\iA$ is
the quotient of $\Spec_\iy \pi_*\iA \smallsetminus V(\pi_*A_+)$
by the action of $\G_m$ on the fibers over $\iy$. It is a coarse quotient
since $\pi_*\iA$ is not necessarily generated in degree~$1$. Likewise,
$\Proj^\pi_\ix \iA$ is the quotient of $\Spec_\ix \iA
\smallsetminus V(\pi^{-1}(\pi_*\iA_+))$ by the action of $\G_m$ on the
fibers over $\ix$.

Since
the property of being a good moduli space is preserved by base change,
$\Spec_\ix \iA_{f} \to \Spec (\pi_*\iA)_{f}$ is a good moduli space
morphism. This gives us the commutative diagram
\[
\xymatrix{%
\Spec_\ix \iA_f\ar[r]^-{q_\ix}\ar[d]^{\pi_{\iA_f}} & \Spec_\ix\iA_{(f)}\ar[d]^{\pi_{\iA_{(f)}}} \\
\Spec_\iy (\pi_*\iA)_f\ar[r]^-{q_\iy} & \Spec_\iy (\pi_*\iA)_{(f)}
}
\]
where ${\pi_{\iA_f}}$ is a good moduli space morphism and $q_\ix$ and $q_\iy$
are coarse $\GG_m$-quotients. Note that the natural transformation $M\to
(q_*q^*M)_0$ is an isomorphism for $q=q_\ix$ and $q=q_\iy$. Since
$(\pi_{\iA})_*$ is compatible with the grading, it follows that
\[
(\pi_{\iA_{(f)}})_*M = ((q_\iy)_*(\pi_{\iA_f})_*(q_\ix)^*M)_0
\]
is a composition of right-exact functors, hence exact. It follows that
$\pi_{\iA_{(f)}}$ is a good moduli space
morphism.
\end{proof}

\begin{remark}[Adequate stacks]
Proposition~\ref{prop.saturatedproj}
also holds for stacks with adequate moduli spaces with essentially identical
arguments.
\end{remark}

\subsection{Properties of saturated Proj and saturated blowups}

\begin{proposition}\label{prop.saturatedproj-and-basechange}
Let $\pi \colon \ix \to \iy$ be a good moduli space morphism and let $\iA$ be a
graded finitely generated $\cO_\ix$-algebra.
\begin{enumerate}
\item\label{item.saturatedproj.closed}
 If $\iA\to \iB$ is a surjection onto another graded $\cO_\ix$-algebra,
  then $\Proj^\pi_\ix \iB = \Proj^\pi_\ix \iA\times_{\Proj_\ix \iA} \Proj_\ix
  \iB$. In particular, there is a closed immersion $\Proj^\pi_\ix \iB\to
  \Proj^\pi_\ix \iA$.
\item\label{item.saturatedproj.basechange}
  If $g\colon \iy'\to \iy$ is any morphism and $f\colon \ix'=\ix\times_\iy
  \iy'\to \ix$ is the pullback with good moduli space $\pi'\colon \ix'\to
  \iy'$, then $\Proj^{\pi'}_{\ix'} f^*\iA=\bigl(\Proj^\pi_\ix \iA\bigr)\times_\ix \ix'$
  as open substacks of $\Proj_{\ix'} f^*\iA$.
\end{enumerate}
\end{proposition}
\begin{proof}
\itemref{item.saturatedproj.closed}
  We have a closed immersion $\Proj_\ix \iB\to \Proj_\ix \iA$ and the
  saturated Proj's are the complements of $V(\pi^{-1}\pi_*\iB_+)$ and
  $V(\pi^{-1}\pi_*\iA_+)$ respectively. Since $\pi$ is cohomologically affine,
  $\pi_*$ preserves surjections. It follows that $\pi^{-1}\pi_*\iA_+\to
  \pi^{-1}\pi_*\iB_+$ is surjective and the result follows.

\itemref{item.saturatedproj.basechange}
  There is an isomorphism $\Proj_{\ix'} f^*\iA=\bigl(\Proj_\ix \iA\bigr)\times_\ix \ix'$
  and the saturated Proj's are the complements of $V(\pi'^{-1}\pi'_*f^*\iA_+)$
  and $V(f^{-1}\pi^{-1}\pi_*\iA_+)$ respectively. These are equal since
  $\pi'^*\pi'_*f^*=\pi'^*g^*\pi_*=f^*\pi^*\pi_*$~\cite[Proposition~4.7]{Alp:13}.
\end{proof}

\begin{proposition}\label{prop.saturatedblowups}
Let $\pi \colon \ix \to \iy$ be a good moduli space morphism and let $\ic
\subset \ix$ be a closed substack with sheaf of ideals $\iI$. Let $f\colon
\ix'=\Bl^\pi_\ic \ix\to \ix$ be the saturated blowup with good moduli space
morphism $\pi'\colon \ix'\to \iy'$ and exceptional divisor $\ie\subset \ix'$.
\begin{enumerate}
\item\label{item.saturatedblowup.gen}
  There exists a positive integer $d$ such that $\bigoplus_{n\geq 0}
  \pi_*(I^{dn})$ is generated in degree~$1$.
\item\label{item.saturatedblowup.gms-blowup}
  For $d$ as above, $\iy'=\Bl_{\pi_*(I^d)} \iy$ and $\pi'^{-1}(\iF)=d\ie$
  where $\iF$ is the exceptional divisor of $\iy'\to \iy$.
\item\label{item.saturatedblowup.excdiv}
  As closed subsets, $\ie=f^{-1}(\ic)$ and
  $f^{-1}\bigl(\pi^{-1}(\pi(\ic))\bigr)$, coincide.
\item\label{item.saturatedblowup.dense}
  The map $f$ induces an isomorphism $\ix'\smallsetminus \ie\to
  \ix\smallsetminus \pi^{-1}\bigl(\pi(\ic)\bigr)$. In particular,
  $\ix\smallsetminus \pi^{-1}\bigl(\pi(\ic)\bigr)\subset \ix'$ is
  schematically dense.
\end{enumerate}
\end{proposition}
\begin{proof}
\itemref{item.saturatedblowup.gen}
By Proposition~\ref{prop.saturatedproj}, $\iy'=\Proj_\iy \pi_*\iA$ where
$\iA=\bigoplus_{n\geq 0} \iI^n$. Since $\pi_*\iA$ is a finitely
generated algebra~\cite[Lemma A.2]{AHR:15}, it is generated in degrees $\leq m$ for some $m$.  If
$d$ is a common multiple of the degrees of a set of generators, e.g., $d=m!$, then
$\pi_*\iA^{(d)}=\bigoplus_{n\geq 0} \pi_*I^{dn}$ is generated in degree
$1$.

\itemref{item.saturatedblowup.gms-blowup}
It follows from \itemref{item.saturatedblowup.gen} that $\iy'=\Proj_\iy \pi_*\iA^{(d)}=\Bl_{\pi_*(\iI^d)}
\iy$. To verify that $\pi'^{-1}(\iF)=d\ie$, we may work locally on $\iy$ and at the
chart corresponding to an element $f\in \Gamma(\ix,\iI^d)$. Then
$\ix'=\Spec_\ix (\bigoplus_{n\geq 0} \iI^{nd})_{(f)}$, $\iy'=\Spec_\iy
(\bigoplus_{n\geq 0} \pi_*\iI^{nd})_{(f)}$, $\iF=V(f)$ and $d\ie=V(f)$.

\itemref{item.saturatedblowup.excdiv} Follows immediately from \itemref{item.saturatedblowup.gms-blowup} since
$f^{-1}\bigl(\pi^{-1}(\pi(\ic))\bigr)=\pi'^{-1}(\iF)$ as sets.

\itemref{item.saturatedblowup.dense}
Since the saturated blowup commutes with flat base change on $\iy$
(Proposition~\ref{prop.saturatedproj-and-basechange}\itemref{item.saturatedproj.basechange}), the map
$f\colon \ix'\to \ix$ becomes an isomorphism after restricting
to $\ix\smallsetminus \pi^{-1}\bigl(\pi(\ic)\bigr)$. But
$f^{-1}(\pi^{-1}\bigl(\pi(\ic)\bigr)=\ie$ by \itemref{item.saturatedblowup.excdiv}.
\end{proof}

\begin{remark}
Proposition~\ref{prop.saturatedblowups}~\itemref{item.saturatedblowup.gms-blowup}
is a generalization of~\cite[Lemma 3.11]{Kir:85}.
Proposition~\ref{prop.saturatedblowups}~\itemref{item.saturatedblowup.excdiv} and~\itemref{item.saturatedblowup.dense} generalize~\cite[Theorem 2.3]{Rei:89}
via Example~\ref{ex.saturated-blowup}.
\end{remark}

\subsection{Tame stacks}
Propositions \ref{prop.saturatedproj} and \ref{prop.saturatedblowups}
are non-trivial statements even when $\ix$
is a tame stack and $\pi \colon \ix \to \bX$ is the coarse space morphism.
In this case,
the saturated Proj coincides with the usual Proj, see proof of
Proposition~\ref{prop.stability.saturatedproj} below. We can thus identify the coarse
space of a blowup along a sheaf
of ideals $\iI$ as $\Proj\bigl(\bigoplus \pi_* \iI^n\bigr)$ and as the blowup in $\pi_*\iI^d$ for sufficiently divisible $d$.

\begin{example}
Let $\ix = [\A^2/\Gmu_2]$ where $\Gmu_2$ acts by $-(a,b) = (-a,-b)$.
The coarse space of $\ix$ is the cone $Y = \Spec [x^2, xy, y^2]$.
Proposition \ref{prop.saturatedproj} says that if we let $X'$ be 
the blowup of $\A^2$ at the origin, then the quotient 
$X'/\Gmu_2$ is $\Proj$ of the graded ring
$\bigoplus S_i$ where $S_i$ is the monomial ideal in the invariant
ring $k[x^2,xy,y^2]$ generated by monomials of degree $\lceil i/2 \rceil$.
This is isomorphic to the blowup of $Y$ in $(x^2,xy,y^2)$.
\end{example}

\subsection{Strict transforms of saturated blowups}
Let $\pi \colon \ix \to \iy$ be a good moduli space morphism, let $\ic \subset
\ix$ be a closed substack and consider the saturated blowup $p\colon
\Bl^{\pi}_\ic \ix\to \ix$. We have seen that $p$ is an isomorphism outside
$\ix\smallsetminus \pi^{-1}\bigl(\pi(\ic)\bigr)$ and that $\ix\smallsetminus
\pi^{-1}\bigl(\pi(\ic)\bigr)$ is schematically dense in the saturated
blowup.
\begin{definition}[Strict transform of saturated blowups] \label{def.saturated-strict-transform}
If $\iz\subset \ix$ is a closed substack, then we let the \emph{strict
  transform} of $\iz$ along $p$ denote the schematic closure of
$\iz\smallsetminus \pi^{-1}\bigl(\pi(\ic)\bigr)$ in $p^{-1}(\iz)$. Similarly,
if $f\colon \ix'\to \ix$ is any morphism, then the strict transform of $f$
along $p$ is the schematic closure of $\ix'\smallsetminus
f^{-1}\pi^{-1}\bigl(\pi(\ic)\bigr)$ in $\ix'\times_\ix \Bl^{\pi}_\ic \ix$.
\end{definition}

\begin{proposition}\label{prop.saturatedblowups-and-basechange}
Let $\pi \colon \ix \to \iy$ be a good moduli space morphism, let $\ic
\subset \ix$ be a closed substack and let $p\colon \Bl^\pi_\ic \ix\to \ix$
be the saturated blowup.
\begin{enumerate}
\item\label{item.saturatedblowup.closed}
If $\iz\subset \ix$ is a closed substack, then
$\Bl^\pi_{\ic\cap \iz} \iz$ is the strict transform of $\iz$ along $p$.
In particular, there is a closed immersion $\Bl^\pi_{\ic\cap \iz} \iz\to
\Bl^\pi_\ic \ix$.
\item\label{item.saturatedblowup.basechange}
If $g\colon \iy'\to \iy$ is any morphism,
and $\ix'=\ix\times_\iy \iy'$ and $\ic'=\ic\times_\ix \ix'$, then
$\Bl^{\pi'}_{\ic'} \ix'$ is the \emph{strict transform} of $\ix'\to \ix$
along $p$. In particular, there is a closed immersion
$\Bl^{\pi'}_{\ic'} \ix'\to \bigl(\Bl^\pi_{\ic} \ix\bigr)\times_\ix \ix'$ and this is
an isomorphism if $g$ is flat.
\end{enumerate}
\end{proposition}
\begin{proof}
\itemref{item.saturatedblowup.closed}
  Let $\iA=\bigoplus \iI^n$ and $\iA'=\bigoplus \iI^n/(\iI^n\cap \iJ)$ where
  $\iI$ defines $\ic$ and $\iJ$ defines $\iz$. Then there is a closed immersion
  $\Bl^\pi_{\ic\cap \iz} \iz=\Proj^\pi_\ix \iA'\to \Proj^\pi_\ix \iA
  =\Bl^\pi_\ic \ix$ by
  Proposition~\ref{prop.saturatedproj-and-basechange}~\itemref{item.saturatedproj.closed}. The result follows since $\Bl^\pi_{\ic\cap \iz} \iz$ equals
  $p^{-1}(\iz)$ outside $\pi^{-1}\bigl(\pi(\ic)\bigr)$ and is schematically
  dense (Proposition~\ref{prop.saturatedblowups}~\itemref{item.saturatedblowup.dense}).

\itemref{item.saturatedblowup.basechange}
Let $\iA=\bigoplus \iI^n$ and $\iA'=\bigoplus \iI'^n$ where $\iI$ defines $\ic$
and $\iI'$ defines $\ic'$. Then $f^*\iA\to \iA'$ is surjective, so there is a
closed immersion $\Bl^{\pi'}_{\ic'} \ix'\to \Bl^\pi_{\ic} \ix\times_\ix \ix'$
(Proposition~\ref{prop.saturatedproj-and-basechange}). The result follows,
since $\Bl^{\pi'}_{\ic'} \ix'\to \ix'$ is an isomorphism outside
$\pi'^{-1}\bigl(\pi'(\ic')\bigr)$ and $\ix'\smallsetminus
\pi'^{-1}\bigl(\pi'(\ic')\bigr)$ is schematically dense in $\Bl^{\pi'}_{\ic'}
\ix'$ (Proposition~\ref{prop.saturatedblowups}~\itemref{item.saturatedblowup.dense}).
When $g$ is flat, then $f^*\iA\to \iA'$ is an isomorphism and the result follows
directly from Proposition~\ref{prop.saturatedproj-and-basechange}.
\end{proof}
\subsection{Saturated Proj and stable points}
Recall that a good moduli space morphism $\ix \stackrel{\pi}
\to \iy$ is stable (resp. properly stable)
if $\ix^s$ is dense (resp. $\ix^{ps}$ is dense).

Here, we give conditions that ensure that if $\ix \stackrel{\pi} \to \iy$
is stable (resp. properly stable) then $\Proj_\ix^{\pi}  \iA \stackrel{\pi\smash{'}}
\to \Proj_\iy \pi_* \iA$ is also stable (resp. properly stable). This result
plays an important role in the proof of our main theorem.

\begin{proposition}\label{prop.stability.saturatedproj}
Let $\pi\colon \ix\to \iy$ be a good moduli space morphism and let $\iA$ be a finitely
generated graded $\cO_\ix$-algebra. Let $f\colon \ix':=\Proj^\pi_\ix \iA\to \ix$
be the saturated Proj and let $\pi'\colon \ix'\to \iy':=\Proj_\iy \pi_*\iA$ be
its good moduli space morphism.
\begin{enumerate}
\item If $\pi$ is properly stable, then $\pi'$ is properly stable.
\item If $\pi$ is stable and $\iA=\bigoplus_{n\geq
  0} \iI^n$ for an ideal $\iI$, then $\pi'$ is stable.
\end{enumerate}
More precisely, in (1), or in (2) under the additional assumption that
$\ix$ is reduced, the inclusion $\ix'\subset \Proj_\ix\iA$ is an equality
over $\ix^s$ and $\ix'^s$ contains $f^{-1}(\ix^s)$. In (2), $\ix'^s$ always
contains $f^{-1}(\ix^s\smallsetminus V(\iI))$.
\end{proposition}
\begin{proof}
The question is smooth-local on $\iy$ so we can assume that $\iy$ is affine.
We can also replace $\iy$ with $\pi(\ix^s)$ and assume that $\ix=\ix^s$,
that is, every stabilizer of $\ix$ has the same dimension.

In the first case, $\pi$ is a coarse moduli space. The induced morphism
$\pi_\iA\colon \Spec_\ix\iA\to \Spec_\iy\pi_*\iA$ is then also a coarse moduli
space. The image along $\pi_\iA$ of $V(\iA_+)$ is $V(\pi_*\iA_+)$. Since
$\pi_\iA$ is a homeomorphism, $\sqrt{\pi^{-1}\pi_*\iA_+}=\sqrt{\iA_+}$. It
follows that $\ix'=\Proj_\ix \iA$.

In the second case, if in addition $\ix$ is reduced, then
$\pi$ factors through a gerbe $g\colon \ix\to \ix_\tame$
and a coarse moduli space $h\colon \ix_\tame\to \iy$
(Proposition~\ref{prop.tamegerbe}). Since $\iI^n=g^*g_*\iI^n$, we conclude
that $\Proj^\pi_\ix \iA=\left(\Proj^h_{\ix_\tame} g_*\iA\right) \times_{\ix_\tame} \ix$
and the question reduces to the first case.

In the second case, without the additional assumption on $\ix$, let
$\iu:=\ix^s\smallsetminus V(\iI)$. Then $\Proj_\ix \iA\to \ix$ is an isomorphism
over $\iu$ and $\iy':=\Proj_\iy \pi_*\iA\to \iy$ is an isomorphism over
$\pi(\iu)$ so $\iu\subset \ix'^s$. Moreover, $\iu\subset \ix'$ is dense so
$\ix'$ is stable.
\end{proof}

The following example shows that when $\pi \colon \ix \to \iy$ is stable
but not properly stable, the condition that $\iA$ is a Rees algebra in (2) is not superfluous.
\begin{example} \label{ex.properstable.necc}
  Let $\ix  =B\G_m$. The good moduli space $\pi \colon \ix \to \iy = \Spec k$
  is stable but not properly stable.
  Let $\iA=\Sym(\chi_0 \oplus \chi_1)$, where $\chi_n$ denotes the line
  bundle on $B\G_m$ corresponding to the character of $\G_m$ with weight $n$,
  and let $\ix' = \Proj_\ix\iA = [\Pro^1/\G_m]$. Then $\Proj_\ix^\pi\iA =
  [(\Pro^1)^{ss}({\mathcal O}_{\Pro^1}(1))/\G_m] = [\Aff^1/\G_m]$ where $\G_m$
  acts on $\Aff^1$ with weight one (cf.\ Example~\ref{ex.saturated-proj}).
  The good moduli space $\Proj_\ix^\pi\iA \to \Spec k$
  is not stable because $\Aff^1$ has no GIT stable points since the generic orbit is not closed in $\Aff^1$.
\end{example}

\subsection{Resolutions of singularities of stacks
with good moduli spaces}

Recall that in characteristic zero, we have functorial resolution of
singularities by blowups in smooth centers~\cite[Theorem 1.1]{BiMi:08}. To be
precise, there is a functor $BR$ which produces, for each reduced scheme $X$ of
finite type over a field of characteristic zero, a resolution of singularities
$BR(X)$ commuting with smooth morphisms. Here $BR(X)=\{X'\to \ldots \to
X\}$ is a sequence of blowups in smooth centers with $X'$ smooth. Also see
\cite[Theorem 3.36]{Kol:07} although that algorithm may involve singular
centers \cite[Example 3.106]{Kol:07}.
See~\cite{Tem:12} for a generalization to quasi-excellent schemes over $\QQ$.

Artin stacks can be expressed as quotients $[U/R]$
of groupoid schemes $\equalizer{s,t\colon R}{U}$ with $s$ and $t$ smooth morphisms.
Thus, the resolution
functor $BR$ extends uniquely to Artin stacks. In particular, for every
reduced Artin stack $\ix$ that is quasi-excellent and of characteristic zero,
there is a projective
morphism $\tilde{\ix} \to \ix$, a sequence of blowups, which is an
isomorphism over a dense open
set. Similarly, if $X$ is a scheme with an action of a group scheme $G$, then
there is a sequence of blowups in $G$-equivariant smooth centers that resolves
the singularities of $X$.

In general if $\ix$ is an Artin stack, then a resolution of singularities
$\tilde{\ix}$ need not have a good moduli space.
However, the theory of saturated blowups implies that $\tilde{\ix}$
contains an open set which has a good moduli space such that the induced map
of good moduli spaces is proper and birational.

\begin{proposition} \label{prop.resolution}
Let $\ix$ be a reduced Artin stack with stable good moduli space
morphism $\ix \stackrel{\pi} \to \bX$. 
Suppose that $\tilde{\ix}\to \ix$ is a projective birational morphism.
Further assume that either
\begin{enumerate}
\item $\ix$ is properly stable, or
\item $\tilde{\ix}\to \ix$ is a sequence of blowups.
\end{enumerate}
Then there exists an open substack $\ix' \subset
\tilde{\ix}$ such that $\ix'$ has a stable good moduli space
$\ix' \to \bX'$ and the induced morphism of good moduli spaces is projective
and birational.
\end{proposition}

\begin{proof}
Since $\tilde{\ix} \to \ix$ is projective we can write $\tilde{\ix}
= \Proj_\ix{\iA}$ for some graded sheaf $\iA$ of finitely generated ${\mathcal
  O}_{\ix}$-algebras. If $\tilde{\ix} \to \ix$ is a blowup, we choose
$\iA$ as the Rees algebra of this blowup. We treat a sequence of blowups
by induction.

Let $\ix' = \Proj^\pi_\ix \iA$.
By Proposition \ref{prop.saturatedproj},
$\ix' \to \bX' =\Proj_\bX \pi_* \iA$ is a good moduli space morphism.
By Proposition~\ref{prop.stability.saturatedproj} it is stable.
If $\tilde{\ix}\to \ix$ is an isomorphism over the open dense subset
$U\subset \ix$ (resp.\ a sequence of blowups with centers outside $U$),
then $\bX'\to \bX$ is an isomorphism over
the open dense subset $\pi(U\cap \ix^s)$.
\end{proof}

\begin{corollary} \label{cor.resolutionofsings}
Let $\ix$ be a reduced Artin stack of finite type over a field of
characteristic $0$ (or merely quasi-excellent over $\QQ$)
with stable good moduli space
$\ix \stackrel{\pi} \to \bX$.
There exists a quasi-projective birational morphism $\ix' \to \ix$
with the following properties.
\begin{enumerate}
\item The stack $\ix'$ is smooth and admits a good moduli space
$\ix' \stackrel{\pi\smash{'}}\to \bX'$.
\item The induced map of moduli spaces $\bX' \to \bX$
is projective and birational.
\end{enumerate}
\end{corollary}
\begin{proof}
Follows immediately from functorial resolution of singularities by a sequence
of blowups, and Proposition \ref{prop.resolution}.
\end{proof}

\subsection{Saturated Proj and semi-stability} \label{sec.saturated-proj-GIT}
The material in the next two sections is not needed in the sequel but provides a connection between the saturated Proj and semi-stability in GIT and its extension to stacks by
Alper \cite[Section 11]{Alp:13}.

Let $\ix \stackrel{p} \to S$ be a morphism from an algebraic stack to a scheme $S$. If $\iL$ is a line bundle on $\ix$ then following Alper \cite[Definition 11.1]{Alp:13} we define the open substack $\ix^{ss}(\iL/S)$ of semi-stable points to be the
set of $x \in |\ix|$ such that there is an affine open set $U \subset S$
and a section $t \in \Gamma(p^{-1}(U), {\iL}^n)$ for
some $n > 0$ such that $p^{-1}(U)_t$ is a cohomologically affine neighborhood of
$x$. 
If $S$ is affine then a
point $x \in |\ix|$ is semi-stable if and only if there
is a section $t \in \Gamma(\ix, \iL^n)$ such that $\ix_t$ is a cohomologically
affine neighborhood of $x$. In this case we will write $\ix^{ss}(\iL)$
instead of $\ix^{ss}(\iL/S)$.

\begin{remark}\label{rem.ss-locus}
The locus $\ix^{ss}(\iL/S)$ has the following interpretation. The line bundle
$\iL$ gives rise to a map $r_\iL\colon \iu\to \Proj_S (\bigoplus_{n\geq 0} p_*
\iL^n)$ where $\iu \subset \ix$ is the open locus where $p^*p_* \iL^n \to
\iL^n$ is surjective for some $n > 0$, or equivalently, for all sufficiently
divisible $n$. If $V\subset \Proj_S (\bigoplus_{n\geq 0} p_* \iL^n)$ is the
largest open subset over which $r_\iL$ is cohomologically affine, then
$\ix^{ss}(\iL/S) = (r_\iL)^{-1}(V)$ and $\ix^{ss}(\iL/S)\to V$ is a good
moduli space morphism.
\end{remark}

Let $f\colon S'\to S$ be a morphism of schemes, let $\ix' = S' \times_S \ix$
and let $g \colon \ix' \to \ix$ be the projection. Then
$g^{-1}(\ix^{ss}(\iL/S))\subset \ix'^{ss}(g^*\iL/S')$ and this inclusion is
an equality if $f$ is flat. This follows from the previous remark and flat
descent of cohomological affineness.
Alper's definition can thus be extended to the case where $S$ is an algebraic
space or an algebraic stack by flat descent.

Let $f\colon \ix'\to \ix$ be a projective morphism and let $\pi\colon \ix\to
\iy$ be a good moduli space morphism. Choose an $f$-ample line bundle $\iL$
on $\ix'$.
Then $\ix'=\Proj_\ix \left(\bigoplus_{n \geq 0} f_*\iL^n\right)$
and the semi-stable locus of $\ix'$ coincides with the saturated Proj:

\begin{proposition} \label{prop.stacky-variation-GIT}
$\ix'^{ss}(\iL/\iy) =\Proj^\pi_\ix \left(\bigoplus_{n \geq 0} f_*\iL^n\right)$
as open substacks of $\ix'$.
\end{proposition}
\begin{proof}
  Working locally in the smooth or fppf topology we may assume that $\iy$
  is affine so $\ix$ cohomologically affine. Since $\iL$ is $f$-ample,
  if $t \in \Gamma(\ix',\iL^n)$
  then $\ix'_t \to \ix$ is affine
  and so $\ix'_t$ is cohomologically affine. That is, $\ix'^{ss}(\iL) = \iu$
  in the notation of Remark~\ref{rem.ss-locus}.
  Thus a point $x$ is in $\ix'^{ss}(\iL)$ if and only
  if the map $f^* \pi^* \pi_* f_* \iL^n \to \iL^n$ is surjective at $x$ for
  all sufficiently divisible $n$.
  Its complement
  is exactly $V(\pi^{-1}(\pi_*( \bigoplus_{n>0} f_*\iL^n)))$ in $\Proj_\ix \bigl(\bigoplus_{n \geq 0} f_* \iL^n\bigr)$.
\end{proof}

\subsection{Composition of saturated Proj} \label{sec.comp_sat_proj}
In this section, which also is not needed in the sequel, we use the
relationship between semi-stability and saturated Proj to prove that the
composition of saturated Projs is a saturated Proj.
This provides a connection to \cite[Theorem 2.1]{Rei:89}.

Suppose that $\ix \stackrel{\pi} \to \iy$ is a good
moduli space morphism and $\ix' \stackrel{f} \to \ix$ is a projective morphism
with $f$-ample line bundle $\iL$ and $\ix'' \stackrel{g} \to \ix'$
is a projective morphism with $g$-ample line bundle $\iM$.
Set $\iM_a := \iM \otimes g^*\iL^a$. 
By \cite[Tag \spref{0C4K}]{stacks-project}
$\iM_a$ is $fg$-ample for all $a \gg 0$. 

We let $\iA := \bigoplus_{n\geq 0} f_* \iL^n$ so that $\ix'=\Proj_\ix \iA$.
Likewise, we let $\iB_a := \bigoplus_{n \geq 0} f_*g_* \iM_a^n$
so that $\ix''=\Proj_\ix \iB_a$ for all $a \gg 0$.
We let $\ix'^{\ss} := \ix'^{\ss}(\iL/\iy) = \Proj_\ix^{\pi} \iA$
and $\iy'=\Proj_\iy \pi_* \iA$ so that we have a good moduli space morphism
$\pi' \colon \ix'^{\ss} \to \iy'$.
Finally, we let $\iC = \bigoplus_{n \geq 0} g_* \iM^n|_{\ix'^{\ss}}$
so that $g^{-1}(\ix'^{\ss})=\Proj_{\ix'^{\ss}} \iC$.

\begin{proposition} \label{prop.composition-saturated-proj}
  For $a \gg 0$, $\Proj^{\pi}_\ix \iB_a = \Proj^{\pi'}_{\ix'^{\ss}} \iC$.
  Equivalently
  $$ (\ix'')^{ss}(\iM_a/\iy) = \left(g^{-1}(\ix'^{ss})\right)^{\ss}(\iM/\iy').$$
\end{proposition}

\begin{example}[Saturated Projs and GIT] \label{ex.saturated-blowup}
  Let $G$ be a linearly reductive group.
  When $\ix \to \iy$ is the good moduli space morphism
  $BG \to \Spec k$ then we are in the situation of classical GIT.
  If $\ix' = [X'/G]$, $\ix'' = [X''/G]$ and $\iL$ and $\iM$ correspond to $G$-linearized ample line bundles $L$ and $M$ on $X'$ and $X''$ respectively, Proposition \ref{prop.composition-saturated-proj}
  gives a precise description of $(X'')^{ss}(M_a)$ which generalizes \cite[Theorem 2.1]{Rei:89}.
  For example let $g \colon X'' \to X'$ be the blowup of $X'$ along a $G$-invariant
    closed subscheme $C$, let $E$ be the exceptional divisor and let
    $K_a = g^*L^a(-E)$. Then for $a \gg 0$, $X''^{ss}(K_a) = \bigl(
    g^{-1}(X'^{ss})(L)\bigr)^{ss}(-E)$ and
    $[X''^{ss}(K_a)/G] = \Bl_{\ic}^{\pi'} [X'^{ss}(L)/G]$
    where $\ic = [(C \cap X'^{ss}(L))/G]$. This can be viewed as an extension
    of \cite[Theorem 2.4]{Rei:89} when $X'$ and $C$ are singular.
    \end{example}

The proof of Proposition \ref{prop.composition-saturated-proj} follows from
the following three steps.
\begin{enumerate}
\item \label{item.comp.one}
  $(\ix'')^{ss}(\iM_a) \cap g^{-1}(\ix'^{ss}) \subset
  \left( g^{-1}(\ix'^{ss}) \right)^{ss}(\iM/\iy')$
  for all $a \in \ZZ$.

\item \label{item.comp.two}
  For $a \gg 0$ then $\left(g^{-1}(\ix'^{ss})\right)^{ss}(\iM/\iy')
  \subset (\ix'')^{ss}(\iM_a)$.

\item \label{item.comp.three}
  For $a \gg 0$,
  $(\ix'')^{ss}(\iM_a) \subset g^{-1}(\ix'^{ss})$.
\end{enumerate}

In the first two steps, we do not need any hypotheses:
$\ix''\to \ix'$ need not be projective,
$\ix'\to \iy$ need not be the composition of a projective morphism
followed by a good moduli space morphism, and $\iL$ and $\iM$ can be
arbitrary line bundles.

\begin{proof}[Proof of \itemref{item.comp.one}]
  Working locally on $\iy$ we may, without loss of generality, assume that
  $\iy$ is affine. Then $x'' \in (\ix'')^{ss}(\iM_a)$ if and only if
  there is a section $t \in \Gamma(\ix'',\iM_a^n)$ such that $(\ix'')_t$
  is a cohomologically affine neighborhood of $x''$.

  If $x'$ is a point of $\ix'^{ss}$ then there is a section $w\in
  \Gamma(\ix',\iL^m)$ such that $\iu':=(\ix')_w$ is a cohomologically affine
  neighborhood of $x'$. Note that the good moduli space of $\iu'$ is the open
  affine subscheme $U'$ of the scheme $\iy'=\Proj_\iy\left( \bigoplus_{m\geq 0}
  \Gamma(\ix',\iL^m) \right)$ given by $w\neq 0$ and that $\iu' = \pi'^{-1}(U')$.

  Since $\iL^m|_{\iu'}$ is trivial, we have that
  $(\iM_a)^m = \iM^m$ on $g^{-1}(\iu')$, so that $\tilde{t}:=t^m|_{g^{-1}(\iu')}
    \in \Gamma(g^{-1}(\iu'), \iM^{mn})$.
  Hence if $x'' \in (\ix'')^{ss}(\iM_a) \cap g^{-1}(\ix'^{ss})$
  then $x'' \in (\pi' \circ g)^{-1}(U')_{\tilde{t}}$.
  Since $U'\to \iy'$ is affine, we have that $(\pi' \circ g)^{-1}(U')_{\tilde{t}}
  = (\pi' \circ g)^{-1}(U') \cap (\ix'')_t$ is cohomologically affine.
  Therefore $x'' \in \left(g^{-1}(\ix'^{ss})\right)^{ss}(\iM/\iy')$.
\end{proof}

\begin{proof}[Proof of \itemref{item.comp.two}]
  The proof is similar to the proof of \cite[Tag \spref{0C4K}]{stacks-project}.
  
  Again we may assume that $\iy$ is affine. Since $\ix'$ is quasi-compact,
  we may find a sufficiently divisible $m$ such that $\ix'^{ss}$
  is covered by a finite number of cohomologically affine open sets
  $(\ix')_w$ with $w \in \Gamma(\ix', \iL^m)$.

  Suppose that $x'' \in \left(g^{-1}(\ix'^{ss})\right)^{ss}(\iM/\iy')$. Then for
  some $n > 0$ there are sections $w \in \Gamma(\ix', \iL^m)$ and
  $t \in \Gamma(g^{-1}(\ix'_w), \iM^n))$
  such that $g^{-1}(\ix'_w)_t$ is a cohomologically affine neighborhood
  of $x''$. Replacing
  $t$ by $t^m$ we may assume that $m | n$.

  The proof of \cite[Tag \spref{01PW}]{stacks-project}
  adapts to algebraic stacks and implies that there is
  a section $t_1 \in \Gamma(\ix'', \iM^n \otimes g^*\iL^{me})$ such
  that $t_1|_{(\ix'')_{g^*w}} = t(g^*w)^e$ for some $e > 0$.

  Note that for any $b > 0$ if we let  $t_2 =t_1 (g^*w)^b$ then
  $\ix''_{t_2} = (\ix''_{g^*w})_t = g^{-1}(\ix'_w)_t$ is a cohomologically affine
  neighborhood of $x''$ by the choice of $t$.

  Since $m | n$, then for any $a \gg 0$ we can find $b > 0$ such that
  $m(e+b) = na$ so we can take $t_2 \in \Gamma(\ix'', \iM_a^n)$. Hence
  $x'' \in (\ix'')^{ss}(\iM_a)$.
\end{proof}

\begin{proof}[Proof of \itemref{item.comp.three}]
  Our goal is to show that for every $a \gg 0$, if $x''\in \ix''$ is such that
  $g(x'')$ is $\iL$-unstable then $x''$ is $\iM_a$-unstable.  
  We can replace $\iM$ by $\iM_b$ for some $b \gg 0$ since
  $(\iM_b)_a = \iM_{b+a}$.
  We can replace $\iL$ and $\iM$ by $\iL^m$ and $\iM^m$ since
  $\iM^m \otimes g^*(\iL^m)^a = \iM_{a}^m$.
  In particular, if we replace $\iM$ by $\iM_b$ for $b \gg 0$
  we may assume that $\iM$ is $fg$-ample.

  Again we may work locally on $\iy$ and assume that it is affine.
  By the local structure theorem \cite[Theorem 13.1(1)]{AHR:19} (cf.~Theorem \ref{thm.localstructuregeneral})
  we may also assume that
  $\ix = [ \Spec A/\GL_n]$  and $\iy = \Spec A^{\GL_n}$.
  Since the morphisms $\ix' \to \ix$ and $\ix'' \to \ix$
  are projective, $\ix' = [X'/\GL_n]$ and $\ix''=[X''/\GL_n]$ where
  $X' \to \Spec A$ and $X'' \to \Spec A$ are $\GL_n$-equivariant morphisms.
  The ample line bundles $\iL$ and $\iM$ correspond to $\GL_n$-equivariant ample
  line bundles $L$ and $M$ on $X'$ and $X''$ respectively. Since we are working over an affine base, there exists $m > 0$ such that $L^m$ and $M^m$ are both very ample. Thus, replacing $\iL$ and $\iM$ by $\iL^m$
  and $\iM^m$ we assume that there are $\GL_n$-invariant closed embeddings
  $X' \subset \Pro^N_A$, $X'' \subset \Pro^M_A$ such that $L$ and $M$
  are the pullbacks of ${\mathcal O}_{\Pro^N_A}(1)$ and ${\mathcal O}_{\Pro^M_A}(1)$ respectively.

  Since the map on global sections $\Gamma(\Pro^N_A, {\mathcal O}_{\Pro^N_A}(k))
  \to \Gamma(X', L^k)$ is surjective for $k \gg 0$,
  $\ix'^{ss} = [((\Pro_A^N)^{ss} \cap X')/\GL_n]$, where $(\Pro_A^N)^{ss}$
  refers to the semi-stable points with respect to ${\mathcal O}_{\Pro^N}(1)$.
Likewise, $(\ix'')^{ss} = [((\Pro_A^M)^{ss} \cap X'')/\GL_n]$.

 The map
 $X'' \to X'$ factors through the projection $\Pro^N \times \Pro^M \to \Pro^N$
 and it suffices to show that for $a \gg 0$ any point $(x',x'') \in \Pro_A^N \times \Pro_A^M$ with $x'$ unstable is ${\mathcal O}_{\Pro^N_A}(a) \boxtimes {\mathcal O}_{\Pro^M_A}(1)$-unstable.

 To do this we can use the Hilbert--Mumford criterion for the geometrically
 reductive group $\GL_n$ as in the proof of \cite[Theorem 2.1a]{Rei:89}.
\end{proof}

\begin{example}
  The following example shows that the condition that $a \gg 0$
  is necessary for \itemref{item.comp.three} to hold. Take $\ix' = \ix = B\G_m$, $\iy= \Spec k$. Denote by $\chi_d$ for $d \in \ZZ$ the character of $\G_m$
  with weight $d$.
    Let $\iL$ be the line bundle on $B\G_m$ corresponding to the
  character $\chi_1$ of weight one. The global sections
  $\Gamma(\ix', \iL^n)$ are zero since, $\chi_1^{\otimes n} = \chi_n$ has
no non-zero invariants. Hence $\ix'^{ss}(\iL) = \emptyset$.

Let $V = \chi_d + \chi_{-d}$ and
 Set $\ix'' = \Proj_\ix(V)$,
 and take $\iM$ to be the tautological line bundle on $\ix''$. The global
 sections $\Gamma(\ix'', \iM^n)$ is the $\G_m$-invariant subspace
 of the $\G_m$-module $\Sym^n(V)$. Denote by $x,y$ the two coordinate
 functions on $V$ with respect to the decomposition $V = \chi_d + \chi_{-d}$.
 Then $\Gamma(\ix'', \iM^{n}) = \langle x^{n/2}y^{n/2} \rangle$ if $n$ is even
 and zero otherwise, and
 $(\ix'')^{ss}(\iM) = \ix'' \smallsetminus \{[0/\G_m] \cup [\infty/\G_m]\}$.

 Likewise $\Gamma(\ix'', (\iM \otimes \iL^a)^n)$ is the
 invariant subspace of the $\G_m$-representation
 $W  = (\Sym^n(\iM \otimes \iL^a))$.
It decomposes as a sum of characters
 $W = \chi_{n(d+a)} + \chi_{n(d+a)-2d} + \cdots + \chi_{n(-d+a)}$. If $a \leq d$
 and $n=2d$ then one of these summands has weight $0$, while if $a > d$ then all summands have positive weights.
Hence
 $(\ix'')^{ss} (\iM_a) = \emptyset$ only when $a > d$. 
 \end{example}

%%%%%%%%%%%%%%%%%%%%%%%%%%%%%%%%%%%%%%%%%%%%%%%%%%%%%%%%%%%%%%%%
\section{Reichstein transforms and saturated blowups} \label{sec.reichstein}
When $\ix$ and $\ic$ are smooth over a field $k$,
then the saturated blowup of $\ix$
along $\ic$ has a particularly nice description in terms of
{\em Reichstein transforms}.

The following definition is a straightforward extension of
the one originally made in \cite{EdMo:12}. 
\begin{definition} \label{def.reichstein}
  Let $\ix \stackrel{\pi} \to \iy$ be a good moduli space morphism and let $\ic \subset
  \ix$ be a closed substack. The {\em Reichstein transform} with center 
$\ic$, is the stack $R(\ix, \ic)$
obtained by deleting the strict transform of the
  saturation $\pi^{-1}(\pi(\ic))$ in the blowup of $\ix$ along~$\ic$.
\end{definition}

Recall that if $f\colon \Bl_\ic \ix\to \ix$ is the blowup, then
$E=f^{-1}(\ic)$ is the exceptional divisor and
$\overline{f^{-1}(\iz) \smallsetminus E}=\Bl_{\ic\cap\iz} \iz$ is the strict transform of
$\iz\subset \ix$.

\begin{remark}
  Observe that if $\ix$ and $\ic$ are smooth, then $R(\ix, \ic)$ is
  smooth since it is an open set in the blowup of a smooth stack along
  a closed smooth substack.
\end{remark}

\begin{remark}\label{rem.reichstein.basechange}
Let
\vspace{-4mm}% HACK
\[
\xymatrix{
\ix^\prime  \ar[r]^\psi \ar[d]^{\pi^{\prime}} & \ix\ar[d]^{\pi}\\
\iy^\prime \ar[r]^\phi & \iy
}
\]
be a cartesian diagram where the horizontal maps are flat and the vertical maps
are good moduli space morphisms. If $\ic \subset \ix$ is a closed substack, then
$R(\ix^\prime, \psi^{-1}\ic) = \ix^\prime \times_\ix R(\ix, \ic)$.
This follows because blowups commute with flat base change and the saturation of
$\psi^{-1}(\ic)$ is the inverse image of the saturation of $\ic$.
\end{remark}

\begin{definition}[Equivariant Reichstein
  transform] \label{def.equivreichstein} If an algebraic group $G$
  acts on a scheme $X$ with a good quotient $p \colon X \to X\gitq G$ and
  $C$ is a $G$-invariant closed subscheme, then we write $R_G(X,C)$
  for the complement of the strict transform of $p^{-1}p(C)$ in
the blowup of $X$ along $C$.  There
  is a natural $G$-action on $R_G(X,C)$ and $R([X/G], [C/G]) =
  [R_G(X,C)/G]$.
\end{definition}

\begin{proposition} \label{prop.reichstein.is.saturatedproj}
Let $\pi \colon \ix \to \iy$ be a good moduli space morphism and let $\ic
\subset \ix$ be a closed substack. If $\ix$ and
$\ic$ are smooth over a field $k$, then $R(\ix, \ic)=\Bl^\pi_\ic \ix$ as open
substacks of $\Bl_\ic \ix$. In general, $\Bl^\pi_\ic \ix\subset R(\ix, \ic)$.
\end{proposition}

\begin{proof}
Let $\iA = \bigoplus_{n \geq 0} \iI^n$ where $\iI$ defines $\ic$.
The saturation of $\ic$ is the subscheme defined by the ideal $\iJ = \pi_* \iI \cdot \cO_\ix$ so 
the strict transform of the saturation is the blowup of
the substack $V(\iJ)$ along the ideal $\iI/\iJ$, which is
$\Proj_\ic\bigl( \bigoplus_{n \geq 0} (\iI^n/(\iI^n \cap \iJ))\bigr)$. 
Thus the ideal of the strict transform of the 
saturation
is the graded ideal $\bigoplus_{n\geq 0} (\iI^n\cap \iJ) \subset \iA$.
We need to show that this ideal defines 
the same closed subset of the points of the blowup as the ideal $\pi^{-1}\pi_*(\iA_+)$.

Since $\iA_+ = \bigoplus_{n>0} \iI^n$ we have that
\[
\pi^{-1}\pi_*(\iA_+) = \pi_*(\iA_+) \cdot \iA
= {\textstyle \bigoplus_{n>0} \iK_n}
\]
where $\iK_n=\sum_{k>0} \pi_*(\iI^k) \iI^{n-k}$. It suffices to show that
\[
\iI^n \cap \iJ = \iK_n \;\;\forall n \gg 0.
\]
Observe that $\pi_*(\iI^k) \cdot \cO_\ix \subset \iI^k$ and $\pi_*(\iI^k) \cdot
\cO_\ix \subset \pi_*(\iI) \cdot \cO_\ix=\iJ$ so $\pi_*(\iI^k) \iI^{n-k}
\subset \iI^n\cap \iJ$. Hence $\iK_n \subset \iI^n \cap \iJ$ for all $n$.

To establish the opposite inclusion, we work smooth-locally on $\iy$. We may
thus assume that $\iy=\Spec A$ and $\pi_*\iI=(f_1,f_2,\dots,f_a)\subset A$.
When $\ix$ and $\ic$ are smooth over a field, the ideal
$\iI^n\cap \iJ$ can locally be described as all functions in
$\iJ=(f_1,f_2,\dots,f_a)\cdot \cO_\ix$ that vanish to order at least $n$ along $\ic$.
If $\ord_\ic(f_i)=d_i$, that is, if $\pi^*f_i\in \iI^{d_i}\smallsetminus \iI^{d_i+1}$,
then for any $n$ greater than all the $d_i$'s, we have
that $\iI^n\cap \iJ=\sum_{i=1}^a f_i\cdot \iI^{n-d_i}$. Since $f_i\in
\pi_*(\iI^{d_i})$ it follows that
\[
\iI^n\cap \iJ\subset \sum_{i=1}^a \pi_*(\iI^{d_i}) \iI^{n-d_i}\subset \iK_n
\]
for all $n$ greater than the $d_i$. 
\end{proof}

The following example shows that if $\ic$ is singular, then the Reichstein
transform need not equal the saturated blowup and could even fail to have a
good moduli space. Also see Examples~\ref{ex.false} and \ref{ex.false2}.

\begin{example}
Let $U=\Spec k[x,y]$ where $\GG_m$ acts by $\lambda(a,b)=(\lambda
a,\lambda^{-1}b)$ and let $\ix=[U/\GG_m]$. Its good moduli space is $\bX =
\Spec k[xy]$. Let $Z=V(x^2y,xy^2)\subset U$ and $\ic=[Z/\GG_m]$. Its saturation
is $\sat \ic=V(x^2y^2)$ which has strict transform $\Bl_\ic \sat
\ic=\emptyset$.  Thus, the Reichstein transform $R(\ix, \ic)$ equals
$\Bl_\ic \ix$.

We will show that $\Bl_\ic \ix$ has no good moduli space. To see this, note
that $(x^2y, y^2x) = (xy) \cdot (x,y)$. Since $(xy)$ is invertible
we conclude
that $\Bl_\ic \ix=\Bl_\iP \ix$ where $\iP=[V(x,y)/\G_m]$. The exceptional divisor of the
latter blowup is $[\PP^1/\GG_m]$ where $\GG_m$ acts by $\lambda[a:b]=[\lambda
  a:\lambda^{-1}b]$. This has no good moduli space since the closure of the
open orbit contains the two fixed points $[0:1]$ and $[1:0]$.

The Reichstein transform $R(\ix, \iP)$, on the other hand, equals $\Bl_\iP \ix
\smallsetminus [(D_1\amalg D_2)/\G_m]$ where $D_1$ and $D_2$ are the strict
transforms of the coordinate axes. In particular, the exceptional divisor is
missing the two points $[0:1]$ and $[1:0]$ and $R(\ix, \iP)$ is tame with coarse moduli space
$\Bl_{\pi(\iP)} \bX=\bX$.
\end{example}

%%%%%%%%%%%%%%%%%%%%%%%%%%%%%%%%%%%%%%%%%%%%%%%%%%%%%%%%%%%%%%%%%%%%%%%%%%%%%%%%%%%%%%

\section{Equivariant Reichstein transforms and fixed points}
In this section all schemes are of finite type over a field
$k$. For simplicity of exposition we will also assume that $k$ is algebraically closed. The goal of this section is to prove the following theorem. 
\begin{theorem} \label{thm.fixed}
Let $X= \Spec A$ be a smooth affine scheme with an
action of a smooth connected linearly reductive group $G$.
Then $R_G(X,X^G)^G  = \emptyset$.
\end{theorem}
\begin{remark}\label{rem.fixlocus.smooth}
By \cite[Proposition A.8.10]{CGP:10} the fixed locus $X^G$ is a closed smooth subscheme of $X$. Note that if $G$ acts trivially, then $X^G=X$ and
$R_G(X,X^G) = \emptyset$.
\end{remark}

\begin{remark}
Theorem \ref{thm.fixed} is false if we drop the assumption that $X$ is smooth
or that $G$ is smooth and connected.
See Examples \ref{ex.false} and \ref{ex.false3} below.
\end{remark}

\subsection{The case of a representation}
In this section we prove Theorem \ref{thm.fixed} when $X=V$ is a representation
of $G$.

\begin{proposition} \label{prop.repfixed}
Let $V$ be a representation of a smooth connected linearly
reductive group $G$. Then $R_G(V,V^G)^G = \emptyset$.
\end{proposition}
\begin{proof}
Decompose
$V = V^0 \oplus V^m$ such that $V^0$ is the trivial submodule and
$V^m$ is a sum of non-trivial irreducible $G$-modules. 
Viewing $V$ as a variety we write $V = V^0 \times V^m$.
The fixed locus for the action of $G$ on $V$ is
$V^0 \times \{0\}$, so the blowup of $V$ along $V^G$ is isomorphic to
$V^0 \times \tilde{V}^m$ where $\tilde{V}^m$ is the blowup 
of $V^m$ at the origin. Also, the saturation of 
$V^G$ is $V^0 \times \sat_G\{0\}$ where $\sat_G\{0\}$ is the
$G$-saturation of the origin
in the representation $V^m$. 
Thus $R_G(V,V^G) = V^0 \times R_G(V^m,0)$ so to prove the proposition 
we are reduced to the case that $V =V^m$; that is, $V$ is a sum of non-trivial
irreducible representations and $\{0\}$ is the only $G$-fixed point.

To prove the proposition we must show that every $G$-fixed point of the exceptional divisor $\Pro(V) \subset \tilde{V}$
is contained in the strict transform of
\[
\sat_G \{0\} = \{ v \in V : 0 \in \overline{Gv} \}.
\]
Let $x \in \Pro(V)$ be a $G$-fixed point. The fixed point $x$ corresponds
to a $G$-invariant line $L \subset V$, inducing a character $\chi$ of $G$. Since
the origin is the only fixed point, the character $\chi$ is necessarily non-trivial. Since $G$ is smooth and connected, there is a
1-parameter subgroup $\lambda$ such that $\langle \lambda, \chi \rangle > 0$.
Then $\lambda$ acts with positive weight $\alpha$ on $L$ and
thus $L \subset \sat_\lambda\{0\} = V_\lambda^+ \cup V_\lambda^-$
where
\begin{align*}
V_\lambda^+ &= \{ v \in V : \lim_{t \to 0} \lambda(t)v =0\},\\
V_\lambda^- &= \{ v \in V : \lim_{t \to \infty} \lambda(t)v =0\}
\end{align*}
are the linear subspaces where $\lambda$ acts with positive weights and
negative weights respectively.

Since $\sat_G \{0\} \supset \sat_\lambda \{0\}$, it suffices to show
that $x \in \Pro(V)$ lies in the strict transform of $\sat_\lambda \{0\}$.
The blowup of $\sat_\lambda \{0\}$ in the origin intersects the exceptional divisor of $\tilde{V}$ in the (disjoint) linear subspaces $\Pro(V_\lambda^+)\cup
\Pro(V_\lambda^-)
\subset \Pro(V)$. Since $L \subset V_\lambda^+$ we see that our fixed point
$x$ is in $\Pro(V_\lambda^+)$ as desired.
\end{proof}

\subsection{Completion of the proof of Theorem \ref{thm.fixed}}
The following lemma is a special case of~\cite[Lemma on
  p.~96]{Lun:73} and Luna's fundamental lemma~\cite[p.~94]{Lun:73}.
For the convenience of the reader, we include the first part of the proof.
\begin{lemma}[Linearization]\label{lem.luna}
Let $X=\Spec A$ be a smooth affine scheme with the action of a linearly
reductive group $G$. If $x\in X^G$ is a closed fixed point,
then there is a $G$-saturated affine neighborhood $U$ of $x$ and
a $G$-equivariant \emph{strongly \'etale} morphism $\phi \colon U\to T_x X$,
with $\phi(x) = 0$. That is,
the diagram
\begin{equation*}
\vcenter{\xymatrix{
U \ar[r]^-\phi \ar[d]^{\pi_U} & T_xX\ar[d]^{\pi}\\
U\gitq G \ar[r]^-\psi & T_x X\gitq G
}}
\end{equation*}
is cartesian and the horizontal arrows are \'etale.
\end{lemma}
\begin{proof}
Let $\mathfrak{m}$ be the maximal ideal corresponding to $x$.
Since $x$ is $G$-fixed the quotient map $\mathfrak{m} \to \mathfrak{m}/\mathfrak{m}^2$
is a map of $G$-modules. By the local finiteness of group 
actions there is a finitely generated $G$-submodule $V \subset \mathfrak{m}$
such that the restriction $V \to \mathfrak{m}/\mathfrak{m}^2$ is surjective. 
Since $G$ is linearly reductive there is a summand $W \subset V$
such that $W \to \mathfrak{m}/\mathfrak{m}^2$ is an isomorphism of $G$-modules.
Since $W \subset A$ we obtain a $G$-equivariant morphism 
$X \to T_xX = \Spec ( \Sym (\mathfrak{m}/\mathfrak{m}^2) )$ which is \'etale at $x$.
Luna's fundamental lemma now gives an open saturated neighborhood $U$ of $x$
such that $U\to T_xX$ is strongly \'etale.
\end{proof}

\begin{remark}
Using Lemma~\ref{lem.luna} and arguing as in the proof of
Proposition~\ref{prop.repfixed}, we recover the result that $X^G$ is smooth
(Remark~\ref{rem.fixlocus.smooth}).
\end{remark}

\begin{proof}[Completion of the Proof of Theorem \ref{thm.fixed}]
  Every $G$-fixed point of $R_G(X,X^G)$ lies in the exceptional divisor
  $\PP(N_{X^G}X)$. To show that $R_G(X,X^G)^G= \emptyset$ we can work
  locally in a neighborhood of a point $x \in X^G$. Thus we may assume
  (Lemma~\ref{lem.luna})
  that there is a strongly \'etale morphism $X \to T_xX$ yielding a
  cartesian diagram
\[
\xymatrix{ X \ar[r]\ar[d] & T_x X\ar[d] \\
X\gitq G \ar[r] & T_xX\gitq G.
}
\]
Hence $R_G(X,X^G)$ can be identified with the pullback of $R_G\bigl(T_xX, (T_xX)^G\bigr)$
along the morphism $X\gitq G \to T_xX\gitq G$
(Remark~\ref{rem.reichstein.basechange}). By Proposition  \ref{prop.repfixed},
$R_G\bigl(T_xX,(T_xX)^G\bigr)^G = \emptyset$ so therefore 
$R_G(X,X^G)^G = \emptyset$ as well.
\end{proof}

\begin{example} \label{ex.false}
Note that the conclusion of Theorem~\ref{thm.fixed} is false without the assumption that
$X$ is smooth. Let $V$ be the 3-dimensional representation of $G ={\mathbb G}_m$ 
with weights $(-1,1,3)$.
The polynomial
$f= x_1x_3^2 + x_2^5$ is $G$-homogeneous of weight $5$, 
so the subvariety $X = V(f)$ is $G$-invariant. Since all weights
for the $G$-action are non-zero $X^G = (\A^3)^G = \{0\}$.

Let $\tilde{\A}^3$ be the blowup of the origin. The exceptional
divisor is $\Pro(V)$ and has three $G$-fixed points $P_0 = [0:0:1], P_1 = [0:1:0],
P_2 =[1:0:0]$.  The exceptional divisor of $\tilde{X}$ is the projectivized
tangent cone $\Pro(C_{\{0\}} X)$. Since $X= V(f)$ is a hypersurface and
$x_1x_3^2$
  is the sole term of lowest degree in $f$, we see that
$\Pro(C_{\{0\}}X)$ is the subscheme $V(x_1x_3^2) \subset \Pro(V)$.
This subvariety contains the 3 fixed points, so $\tilde{X}$ has 3 fixed points.

The saturation of $0$ in $X$ with respect to the $G$-action is 
$\left(X \cap V(x_1)\right) \cup \left(X \cap V(x_2,x_3)\right)$. 
The intersection of the exceptional divisor with the strict transform
of $X \cap V(x_1)$ is the projective subscheme $V(x_1, x_2^5)$ whose reduction
is $P_0$.

The intersection of the exceptional divisor with the strict transform
of $X \cap V(x_2,x_3)$ is the point $V(x_2,x_3) = P_2$.
Thus the strict transform of the saturation of $0$ in $X$ does not contain
all of the fixed points of $\tilde{X}$. Hence $R_G(X,X^G)^G \neq \emptyset$.

The exceptional divisor of the saturated blowup of $X$ in the origin is
$\Pro(C_{\{0\}}X)\smallsetminus \bigl(V(x_1)\cup V(x_2,x_3)\bigr)=V(x_3^2)\smallsetminus \bigl(V(x_1)\cup V(x_2)\bigr)$ which has no $G$-fixed points.
\end{example}

\begin{example} \label{ex.false2}
Consider the following variation of Example~\ref{ex.false}.  Let $V$ be the
5-dimensional representation of $G ={\mathbb G}_m$ with weights $(-1,1,3,2,7)$.
The polynomials $f= x_1x_3^2 + x_2^5$ and $g=x_1x_5 + x_4^3$ are
$G$-homogeneous of weights $5$ and $6$ so the subvariety $X = V(f,g)$ is
$G$-invariant. As before we blow up the origin and the exceptional divisor of
$\tilde{\A}^5$ is $\Pro(V)$ which has five $G$-fixed points. The
exceptional divisor of $\tilde{X}$ is $\Pro(C_{\{0\}} X)$ which is given by
$V(x_1x_3^2,x_1x_5)\subset \Pro(V)$. It contains the five fixed points of
$\Pro(V)$. The saturation of $0$ in $X$ is the union of $X\cap
V(x_1)=V(x_1,x_2^5,x_4^3)$ and $X\cap V(x_2,x_3,x_4,x_5)=V(x_2,x_3,x_4,x_5)$.

In particular, the exceptional divisor of $R_G(X,X^G)$ contains the closed subscheme
$V(x_1,x_3,x_5)=\Pro(W)$ where $W$ is the 2-dimensional representation with
weights $(1,2)$. But $\Pro(W)$ admits no good quotient by $G$ since the closure of the open orbit contains both $G$-fixed points. It follows that $R_G(X,X^G)$ does
not admit a good quotient.

The exceptional divisor of the saturated blowup of $X$ in the origin is
$\Pro(C_{\{0\}}X)\smallsetminus \bigl(V(x_1)\cup
V(x_2,x_3,x_4,x_5)\bigr)=V(x_3^2,x_5)\smallsetminus \bigl(V(x_1)\cup
V(x_2,x_4)\bigr)$ which has no $G$-fixed points and the saturated blowup
admits a good moduli space.
\end{example}

\begin{example} \label{ex.false3}
The assumption that $G$ is smooth and connected is needed in
Theorem~\ref{thm.fixed} and Proposition~\ref{prop.repfixed}. Indeed, for any
integer $r\geq 2$, let $G=\Gmu_r$ act on the standard $1$-dimensional
representation $V$. Then $R_G(V,V^G)\to V$ is an isomorphism.
\end{example}

%%%%%%%%%%%%%%%%%%%%%%%%%%%%%%%%%%%%%%%%%%%%%%%%%%%%%%%%%%%%%%%%%%%%%%%%%%%%%

\section{The proof of Theorem \ref{thm.main} in the smooth case}\label{sec.proof.smooth}

In this section we prove the main theorem when $\ix$ is smooth over a
base scheme $S$ and also prove that the
algorithm is functorial with respect to strong morphisms.

\subsection{Proof of Theorem \ref{thm.main} when $\ix$ is smooth over an algebraically closed field}
Let $\ix$ be a smooth stack with a stable good moduli space and let $\ie$
be an snc divisor (e.g., $\ie=\emptyset$).
Taking connected components, we may assume that $\ix$ is irreducible.
By Lemma \ref{lem.maxlocus}, for any stack $\ix$ the locus
$\ix^{\maxlocus}$ of points of $\ix$ with maximal
dimensional stabilizer is a closed subset of $|\ix|$. Moreover, if
$\ix$ is smooth, then Proposition \ref{prop.ixmax} implies that $\ix^{\maxlocus}$
with its reduced induced substack structure is also smooth.
When $\ix = [X/G]$ with $X$ smooth, $G$ linearly reductive and $X$ has a $G$-fixed point,
then $\ix^{\maxlocus} = [X^{G_0}/G]$ where $G_0$ is the reduced identity component
of $G$. For an arbitrary smooth stack $\ix$
with good moduli space $\bX$ the substack
$\ix^{\maxlocus}$ can be \'etale locally described as follows.

If $x$ is a closed point with stabilizer $G_x$, then
by \cite[Theorem 4.12]{AHR:15}
there is
a cartesian diagram of stacks and good moduli spaces
\[
\xymatrix{%
[U/G_x]\ar[r]\ar[d] & \ix\ar[d]^{\pi}\\
U\gitq G_x\ar[r] & \bX\ar@{}[ul]|\square
}
\]
where the horizontal maps are \'etale. In  this setup, if $\dim(G_x)$
is maximal, then the inverse image
of $\ix^{\maxlocus}$ in $[U/G_x]$ is $[U^{G_0}/G_x]$ where $G_0$ is the reduced identity component of $G_x$.
See Appendix \ref{app.fixed} for more details.

The proof of Theorem \ref{thm.main} proceeds by induction on the maximum
stabilizer dimension. First suppose that the maximum stabilizer dimension equals
the minimum stabilizer dimension. Then $\ix^s=\ix$ and $\ix$ is a gerbe over a
tame stack $\ix_\tame$ by Proposition \ref{prop.tamegerbe}. If $\ix$ is properly stable,
then every stabilizer has dimension zero and $\ix=\ix_\tame$ is a tame stack by
Proposition \ref{prop.tameisgood}. We have thus shown that the sequence of
length $0$, that is $\ix_n=\ix_0=\ix$, satisfies  conclusions
\itemref{item.mainthm.final.stable}--\itemref{item.mainthm.final.gerbe} of 
Theorem \ref{thm.main}.

If the maximum stabilizer dimension is greater than the minimum stabilizer
dimension, we let $\ix_0=\ix$, $\ic_0=\ix^\maxlocus$ and $f_1\colon
\ix_1=R(\ix_0,\ic_0)\to \ix_0$. The following Proposition shows that
the conclusions in Theorem \ref{thm.main} hold for $\ell=0$. In particular, $(\ix_1, \ie_1)$ satisfies the
hypothesis of Theorem~\ref{thm.main} and the maximal stabilizer dimension of
$\ix_1$ has dropped. By induction, we thus have $\ix_n\to \ldots \to \ix_1$
such that the conclusions also hold for $\ell=1,\dots,n$ and Theorem \ref{thm.main}
for $\ix$ follows.

\begin{proposition} \label{prop.inductionstep}
  Let $\ix$ be a smooth irreducible Artin stack with good moduli space morphism
$\pi\colon \ix \to \bX$ and let $\ie$ be an snc divisor on $\ix$.
Let $f\colon \ix' = R(\ix, \ix^\maxlocus)\to \ix$ be the Reichstein transform.
\begin{enumerate}
\myitemi{1a}\label{item.reichstein.center}
  $\ic:=\ix^\maxlocus$ is smooth and meets $\ie$ with normal crossings.
\myitemi{1b}\label{item.reichstein.smooth}
  $\ix'$ and $f^{-1}(\ic)$ are smooth and
  $\ie':=f^{-1}(\ic)\cup f^{-1}(\ie)$ is snc.
\myitemi{2a}\label{item.reichstein.stablegms}
The stack $\ix'$ has a good moduli space $\bX'$.
If $\pi$ stable (resp. properly stable), then so is
$\pi' \colon \ix' \to \bX'$.
\myitemi{2b}\label{item.reichstein.isolocus}
  $f$ induces an isomorphism $\ix'\smallsetminus f^{-1}(\ic)\to \ix\smallsetminus \pi^{-1}\bigl(\pi(\ic)\bigr)$. In particular, $f$ induces an isomorphism
$\ix'^s\smallsetminus f^{-1}(\ic)\to \ix^s$.
\myitemi{2c}\label{item.reichstein.gms-isolocus}
The induced morphism of good moduli spaces $\bX' \to \bX$
is 
projective
and an isomorphism over $\bX\smallsetminus \pi(\ic)$. In particular,
if $\ix^\maxlocus \neq \ix$ then it is an isomorphism over $\bX^s$.

\myitemi{3}\label{item.reichstein.stabdim-drops} Every point of $\ix'$
has a stabilizer of dimension strictly less than the maximum dimension of the
stabilizers of points of $\ix$. 
\end{enumerate}
\end{proposition}
\begin{proof}
  Assertion \itemref{item.reichstein.center} is Proposition \ref{prop.ixmax}.
  The Reichstein transform $\ix'$ is an open substack of $\Bl_{\ix^{\maxlocus}}
  \ix$. Assertion \itemref{item.reichstein.smooth} thus follows from
  the corresponding properties of $\Bl_{\ix^{\maxlocus}}\ix$.

  Since $\ix$ and $\ix^{\maxlocus}$ are smooth, Proposition \ref{prop.reichstein.is.saturatedproj} implies that $R(\ix,\ix^{\maxlocus})$ is the saturated
  blowup of $\ix$ along $\ix^{\maxlocus}$. Assertions \itemref{item.reichstein.stablegms}--\itemref{item.reichstein.gms-isolocus} then
  follow from the properties of the saturated blowup (Propositions \ref{prop.saturatedblowups} and \ref{prop.stability.saturatedproj}) and the fact that
  $\ix^{\maxlocus} \subset \ix\smallsetminus \ix^s$ by Proposition \ref{prop.stable}.

  We now prove assertion~\itemref{item.reichstein.stabdim-drops}.
  By the local structure theorem  \cite[Theorem 4.12]{AHR:15} we may assume
  $\ix = [U/G]$ where $U$ is a smooth affine scheme with a $G$-fixed point
  and $G$
  is a linearly reductive group (Remark~\ref{rem.smoothness}).
Let $G_0$ be the reduced identity component of $G$. Then
$[U/G]^\maxlocus = [U^{G_0}/G]$. To complete the proof we need to show that
$R_{G}(U,U^{G_0})$ has no $G_0$-fixed point. By Theorem \ref{thm.fixed}
we know that $R_{G_0}(U,U^{G_0})$ has no $G_0$-fixed points. We will prove~\itemref{item.reichstein.stabdim-drops} by showing that $R_{G}(U, U^{G_0}) = R_{G_0}(U,U^{G_0})$ as open subschemes of the
blowup of $U$ along $U^{G_0}$.

Consider the maps of quotients $U \stackrel{\pi_0} \to U \gitq G_0 \stackrel{q}
\to U \gitq G$. If $U = \Spec A$,
then these maps are induced by the inclusions of rings
\[
A^{G}=(A^{G_0})^{(G/G_0)} \hookrightarrow A^{G_0} \hookrightarrow A.
\]
Since the quotient group $G/G_0$
is a finite $k$-group scheme,
$U\gitq G_0 = \Spec A^{G_0} \to U\gitq G = \Spec (A^{G_0})^{(G/G_0)}$
is a geometric quotient.

If
$C \subset U$ is a $G$-invariant closed subset of $U$, then its image in $U\gitq G_0$ is $(G/G_0)$-invariant, so it is saturated with respect to the quotient
map $U\gitq G_0 \to U \gitq  G$. Hence, as closed subsets
of $U$, the saturations of $C$ with respect to either the quotient map $U \to U \gitq
G_0$ or to $U \to U \gitq G$ are the
same\footnote{The saturations with respect to the quotient maps come with natural scheme structures which are not the same.
  If $I \subset A$ is the ideal defining $C$ in $U$, then the saturation
  of $C$ with respect to the quotient map $U \to U \gitq G_0$ is the ideal
  $I^{G_0}A$ while the ideal defining the saturation of $C$ with respect to the
  quotient map $U \to U \gitq G$ is the ideal $I^{G}A$. While $I^{G}
  A \subset I^{G_0}A$, these ideals need not be equal.}.
It follows that if $C \subset U$
is $G$-invariant, then $R_{G}(U,C)$ and $R_{G_0}(U,C)$ define the same open subset of the blowup of $U$ along $C$. Since $U^{G_0}$ is $G$-invariant
we conclude that $R_{G_0} (U,U^{G_0}) = R_{G}(U, U^{G_0})$ as open subschemes of the blowup.
\end{proof}

\begin{remark}
When $k$ is merely perfect, then it follows by descent that $\ix^\maxlocus$
with its reduced induced substack structure is smooth. The conclusions of the
theorem can then be verified in this case after passing to the algebraic closure of $k$.
\end{remark}

\begin{remark}\label{rem.smoothness}
Let $\ix=[U/G]$ be a smooth Artin stack with good moduli space $\bX$. If $G$ is
not smooth, then it is not automatic that $U$ is smooth. If $u\in U$ is a
closed point fixed by $G$, then the fiber $U\times_\ix BG=\{ u \}$ is regular
and $U\times_\ix BG\to U$ is a regular closed immersion since $\ix$ is
smooth. It follows that $U$ is smooth over an open $G$-invariant neighborhood
of $u$. After replacing $\ix$ with an open saturated neighborhood of the image
of $u$ \cite[Lemma 3.1]{AHR:15}, we can thus assume that $U$ is smooth.
\end{remark}

When $\pi \colon \ix \to \bX$ is a stable good moduli space
morphism, the morphism $\bX'\to \bX$ of Proposition \ref{prop.inductionstep}
is an isomorphism over $\bX^s$, hence birational. When $\pi$ is not stable,
it may happen that the saturation of
$\ix^\maxlocus$ equals $\ix$ and thus that $\ix'=\emptyset$. The following
examples illustrate this.
\begin{example}
Let $\G_m$ act on $X=\A^1$ with weight $1$. The structure map $\A^1 \to \Spec k$
is a good quotient, so $\Spec k$ is the good moduli space of $\ix = [\A^1/\G_m]$.
The stabilizer of any point of $\A^1 \smallsetminus \{0\}$ is trivial, so 
$\ix^\maxlocus= [\{0\}/\G_m]$ and the saturation of $\ix^\maxlocus$ is all of $\ix$.
Hence $R(\ix, \ix^\maxlocus) = \emptyset$.
\end{example}
\begin{example} Here is a non-toric example. 
Let $V=\mathfrak{sl}_2$ be the adjoint representation
  of $G=\SL_2( \CC)$. Explicitly, $V$ can be identified with the vector space
  of traceless $2 \times 2$ matrices with $\SL_2$-action given by
  conjugation. Let $V^\reg \subset V$ be the open set corresponding
  to matrices with non-zero determinant
  and set $X = V^\reg \times \A^2$.
  Let $\ix = [X/G]$ where $G$ acts by conjugation on the first
  factor and translation on the second factor. The map of affine varieties $X
  \to \A^1\smallsetminus \{0\}$ given by $(A,v) \mapsto \det A$ is a
  good quotient, so $\pi \colon \ix \to \A^1 \smallsetminus \{0\}$ is
  a good moduli space morphism. However, the morphism $\pi$ is not
  stable because the only closed orbits are the orbits of pairs
  $(A,0)$.

The stabilizer of a point $(A,v)$ with $v \neq 0$ is trivial and the stabilizer
of $(A,0)$ is conjugate to $T = \diag \left(t \;\; t^{-1}\right)$ and
$\ix^\maxlocus = [(V^\reg \times \{0\})/G]$. 
Thus, $\pi^{-1}(\pi(\ix^\maxlocus)) =\ix$ and $R(\ix, \ix^\maxlocus) = \emptyset$.
\end{example}

\subsection{Completion of the proof when $\ix$ is smooth over a base $S$}\label{ssec.smooth.over.base}
Suppose that $\ix$ is smooth over an algebraic space $S$.
We may assume that $\ix$ is connected.
Let $n$ be the maximum dimension of a stabilizer in $\ix$.
The closed substack $\ix^\maxlocus$ is smooth over $S$ and
has normal crossings with any snc divisor
(Proposition~\ref{prop.ixmax-singular}).
If $\ix^\maxlocus=\ix$, then we are done. Otherwise $\ix^\maxlocus$
does not intersect
the stable locus
and we let $\ix'$ be the saturated blowup of $\ix$ in $\ix^\maxlocus$.

Since $\ix^\maxlocus$ is smooth, the center
$\ix^\maxlocus\subset \ix$ is a transversal regular embedding relative to $S$.
The usual blowup $\Bl_{\ix^\maxlocus} \ix$ is thus smooth over $S$ and
commutes with arbitrary base change $T\to S$ \cite[Proposition 19.4.6]{EGA4}.
It follows that $\ix'\to S$ is smooth and that $\ix'\times_S T =
\Bl^\pi_{\ix^\maxlocus\times_S T} (\ix\times_S T)$ by
Proposition~\ref{prop.saturatedproj-and-basechange}\itemref{item.saturatedproj.basechange}. From the definition of $\ix^\maxlocus$ we have that
$\ix^\maxlocus\times_S T = (\ix\times_S T)^\maxlocus$ if $\ix\times_S T$ has
maximal stabilizer dimension $n$ and $\ix^\maxlocus\times_S T=\emptyset$
otherwise.
The formation of $\ix'$
thus commutes with arbitrary base change $T\to S$. By definition, we also have
that the stable locus satisfies $\ix^s\times_S T = (\ix\times_S T)^s$.

To prove that the maximum dimension of a
stabilizer in $\ix'$ is strictly smaller than $n$, we may thus replace $S$ with
a geometric point $\overline{s}=\Spec \overline{k}$ and $\ix$ with
a connected component of $\ix_{\overline{s}}$. If $\ix$ has no stabilizer
of dimension $n$, then $\ix'=\ix$ and the result is trivial.
If $\ix$ has a stabilizer of dimension $n$, then $\ix^\maxlocus\neq \ix$ because
$\ix^\maxlocus=\ix$ would imply that $\ix=\ix^s$ but
$\ix^\maxlocus\cap \ix^s=\emptyset$.
The result now follows from
Proposition~\ref{prop.inductionstep}.

\begin{remark}
It follows from the proof that when $\ix\to S$ is smooth, the algorithm
of Theorem~\ref{thm.main} commutes with arbitrary base change $T\to S$.
It also follows that if $\ix$ is stable then $\ix\times_S T$ is stable.
\end{remark}

\begin{remark}\label{rem.algorithm.notstable}
  When $\ix$ is smooth over $S$ but $\ix \to \bX$ is not stable, our inductive 
  proof using Proposition \ref{prop.inductionstep} gives a sequence of birational saturated
  blowups $\ix_n \to \ldots \to \ix$ as in Theorem 2.11 but where the final stack
  $\ix_n$ is not a gerbe over a tame stack. Instead $(\ix_n)^{\maxlocus}$
is a gerbe over a tame stack with coarse moduli space $\bX_n$. The algorithm terminates when $\pi_n(\ix_n^{\maxlocus}) = \bX_n$.
\end{remark}

\subsection{Functoriality for strong morphisms}
Let $\ix$ and $\iy$ be Artin stacks with good 
moduli space morphisms, $\pi_\iy \colon \iy \to \bY$, and
$\pi_\ix \colon \ix \to \bX$. Let $f \colon \iy \to \ix$ be a morphism
and let $g \colon \bY \to \bX$ be the induced morphism of good moduli spaces.
\begin{definition}
  We say the morphism $f$ is {\em strong} if
the diagram
$$\xymatrix{\iy \ar[r]^f \ar[d]^{\pi_\iy} & \ix \ar[d]^{\pi_\ix}\\
\bY \ar[r]^g & \bX}$$
is cartesian.
\end{definition}
Note that a strong morphism is representable and stabilizer-preserving.
We thus have an equality
$\iy^\maxlocus=f^{-1}(\ix^\maxlocus)$ of closed substacks by
Proposition~\ref{prop.ixmax-singular} if $\iy$ and $\ix$ have the same
maximal stabilizer dimension.
A sharp criterion for when a morphism is strong can be found in~\cite{Rydh:15}.

\begin{theorem} \label{thm.reichsteinfunctorial}
  Let $f \colon \iy \to \ix$ be a strong morphism of smooth stacks over an
  algebraic space $S$ with
stable good moduli space morphisms $\iy \to \bY$ and
$\ix \to \bX$. Let $\iy'$ and $\ix'$ be the stacks
produced by Theorem \ref{thm.main}. Then there is a natural morphism 
$f' \colon \iy' \to \ix'$ such that the diagram
\vspace{-3mm}% HACK
\[
\xymatrix{\iy' \ar[d] \ar[r]^{f'} & \ix' \ar[d]\\
\iy \ar[r]^f & \ix}
\]
is cartesian.
\end{theorem}
\begin{proof*}
  We may assume that $\ix$ and $\iy$ are connected. Then $\ix'$ and $\iy'$ are
  produced by taking repeated saturated blowups along $\ix^\maxlocus$ and
  $\iy^\maxlocus$ respectively. If $f^{-1}(\ix^\maxlocus)=\emptyset$, then
  $\iy\times_\ix \Bl^{\pi_\ix}_{\ix^\maxlocus}\ix=\iy$. If $f^{-1}(\ix^\maxlocus)\neq
  \emptyset$, then the maximal stabilizer dimensions of $\ix$ and $\iy$
  coincide and $\iy^\maxlocus=f^{-1}(\ix^{\maxlocus})$ since $f$ is strong.
  Hence by Proposition~\ref{prop.saturatedblowups-and-basechange}\itemref{item.saturatedblowup.basechange},
  we have a closed
  immersion $i\colon \Bl^{\pi_\iy}_{\iy^\maxlocus}\iy\to
  \Bl^{\pi_\ix}_{\ix^\maxlocus}\ix\times_\ix \iy$.

  As we saw in the proof of the main theorem
  (Section~\ref{ssec.smooth.over.base}), both saturated blowups commute with
  passage to fibers over $S$. By the fiberwise criterion of flatness, it is
  thus enough to prove that $i$ is an isomorphism after passing to
  fibers. This is the content of the following proposition.
\end{proof*}

\begin{proposition}
  Let $ f \colon \iy \to \ix$ be a strong morphism of smooth stacks defined over a field $k$.
   Then there is a natural morphism
  $f'\colon R(\iy, f^{-1}(\ix^{\maxlocus}))
 \to R(\ix,\ix^{\maxlocus})$
such that the diagram
\vspace{-3mm}% HACK
\begin{equation}\label{diag.functreich}
\vcenter{%
\xymatrix{
R(\iy, f^{-1}(\ix^{\maxlocus})) \ar[d] \ar[r]^-{f'} & R(\ix, \ix^{\maxlocus})\ar[d]\\
\iy \ar[r]^f & \ix
}}
\end{equation}
is cartesian.
\end{proposition}
\begin{proof}
We will prove that
$(\Bl_{\ix^{\maxlocus}}\ix)\times_\ix \iy=\Bl_{f^{-1}(\ix^{\maxlocus})}\iy$ and that
the open subsets $R(\ix,\ix^{\maxlocus})\times_\ix \iy$
and $R(\iy,f^{-1}(\ix^{\maxlocus}))$ coincide. This gives a natural morphism
$f'$ such that \eqref{diag.functreich} is cartesian. These claims can be
verified \'etale or flat locally on $\bX$ and $\bY$ at points of
$\pi_\iy(f^{-1}(\ix^\maxlocus))$ and we may assume that $k$ is algebraically closed. Let $x\in |\ix^\maxlocus|$ and $y\in f^{-1}(x)$.

Since $\bY \to \bX$ is of finite type we can, locally around $\pi_\iy(y)$,
factor it 
as $\bY \hookrightarrow \bX \times \A^n \to \bX$ where the first map is a closed immersion and the second map is the smooth projection.
By base change, this gives a local factorization of the morphism $f$ as 
$\iy \stackrel{i} \hookrightarrow \ix \times \A^n \stackrel{p} \to \ix$.

Since $\ix \times \A^n \to \ix$ is flat it follows from Remark \ref{rem.reichstein.basechange}
that $R(\ix \times \A^n, \ix^{\maxlocus} \times \A^n) = R(\ix, \ix^{\maxlocus})
\times \A^n$.
We are therefore reduced to the case that the map $f$ is a closed immersion.
Since $\ix$ and $\iy$ are smooth, $f$ is necessarily a regular embedding.

We can apply Theorem \cite[Theorem 4.12]{AHR:15} to reduce to
the case that $\ix=[X/G]$ where $G=G_y=G_x$ is a linearly reductive group and
$X$ is an affine scheme. Let $Y=X\times_\ix \iy$. Then $Y\to X$ is a
regular closed immersion and $\iy=[Y/G]$. After passing to an open neighborhood
of $\pi_\iy(y)$ we can assume that $X$ and $Y$ are smooth (Remark~\ref{rem.smoothness}).

Since $T_xX=T_yY\times N_y(Y/X)$, we can slightly modify Lemma
\ref{lem.luna} to obtain the following commutative diagram
(after further shrinking of $\ix$) where the vertical maps are strongly
\'etale and the horizontal maps are strong closed immersions:
\[
\xymatrix{\iy \ar[d] \ar[r]^{f} & \ix \ar[d]\\
[T_y / G] \ar[r] & [T_x / G]. }
\]
Since $f$ is strong, the action of the stabilizer
$G$ is trivial on the normal space $N_y$. Thus
$(T_x)^\maxlocus=(T_y)^\maxlocus\times N_y$ and $T_x \gitq G = T_y \gitq G
\times N_y$. The result is now immediate.
\end{proof}
%%%%%%%%%%%%%%%%%%%%%%%%%%%%%%%%%%%%%%%%%%%%%%%%%%%%%%%%%%%%%%%%%%%%%%%
\section{Corollaries of Theorem \ref{thm.main} in the smooth case}
\subsection{Reduction to quotient stacks}
Suppose that $\ix$ is a smooth Artin stack over a field $k$ such
that the good moduli space morphism $\pi \colon \ix \to \bX$ is properly stable.
The end result of our canonical reduction of stabilizers (Theorem~\ref{thm.main})
is a smooth tame stack $\ix_n$. 
\begin{proposition} \label{prop.quotientstack}
  Let $\ix$ be a smooth Artin stack with properly stable good moduli space.
  Suppose that $\ix_n$ is Deligne--Mumford (automatic if $\characteristic k=0$) and
  that either $\ix$ has generically trivial stabilizer or $\bX$ is
  quasi-projective. Then

\begin{enumerate}
    \item 
$\ix_n$ is a quotient stack $[U/\GL_m]$ where $U$ is an algebraic space.

\item If, in addition, $\bX$ is separated, then $U$ is separated and the action of $\GL_m$
  on $U$ is proper.

  \item If, in addition, $\bX$ is a scheme, then so is $U$.

  \item If, in addition, $\bX$ is a separated
    scheme, then we can take $U$ to be quasi-affine.

  \item If, in addition, $\bX$ is projective, then there is a projective variety $X$
    with a linearized action of a $\GL_m$ such that $X^{s} = X^{ss} =U$. Moreover,
    if $\characteristic k = 0$, we can take $X$ to be smooth.
\end{enumerate}
    \end{proposition}
\begin{proof}
If the generic stabilizer of $\ix$ is trivial, so is the generic
stabilizer of $\ix_n$. Hence by \cite[Theorem 2.18]{EHKV:01} (trivial generic
stabilizer) or \cite[Theorem 2]{KrVi:04} (quasi-projective coarse space), $\ix_n$ is
a quotient stack. This proves (1).

If $\bX$ is separated, then $\ix_n$ is a separated quotient stack
so $\GL_m$ must act properly. This proves (2). (Note that if $\GL_m$ acts properly on $U$,
then $U$ is necessarily separated. This also follows immediately since $U\to \ix_n$ is affine.)

The morphism $U\to \bX_n$ is affine. Indeed, there is a finite surjective
morphism $V\to \ix_n$ \cite[Theorem 2.7]{EHKV:01} where $V$ is a scheme and
$V\to \bX_n$ is finite and surjective, hence affine. It follows that
$U\times_{\ix_n} V\to \bX_n$ is affine and hence $U\to \bX_n$ is affine as well
(Chevalley's theorem). One can also deduce this directly from $U\to \bX_n$
being representable and cohomologically affine (Serre's theorem).

In particular, if $\bX$ is a scheme, then so is $\bX_n$ and $U$. This proves
(3). Similarly, if $\bX$ is a separated scheme, then so is $\bX_n$ and $U$.
But $U$ is a smooth separated scheme and thus has a $G$-equivariant
ample family of line bundles. It follows that $\ix_n$ has the resolution
property and that we can choose $U$ quasi-affine,
see \cite[Theorems 1.1, 1.2]{Tot:04} for further details. This proves (4).

We now prove (5). Since $U$ is quasi-affine, it is also quasi-projective. By
\cite[Theorem 1]{Sum:74} there is an immersion $U \subset \Pro^N$ and a
representation $\GL_m\to\PGL_{N+1}$ such that the $\GL_m$-action
on $U$ is the restriction of the $\PGL_{N+1}$-action on $\Pro^N$.
Let $X$ be the closure of $U$ in $\Pro^N$. The action of $G$ on $X$ is linearized with respect to the line bundle ${\mathcal O}_X(1)$.
Our statement follows from \cite[Converse 1.13]{MFK:94}.

Finally, if $\characteristic k =0$, then by equivariant resolution of singularities we can
embed $U$ into a non-singular projective $G$-variety $X$. 
\end{proof}

Note that we only used that $\bX_n$ is Deligne--Mumford to deduce that
$\ix_n$ is a quotient stack.

\subsection{Resolution of good quotient singularities}
Combining the main theorem with destackification of tame stacks~\cite{Ber:17,BeRy:14}, we
obtain the following results, valid in any characteristic and
also in mixed characteristic since
we may work over any base scheme or algebraic space $S$.
\begin{corollary}[Functorial destackification of stacks with good moduli spaces]\label{cor.res-good-quot-sing}
Let $\ix$ be a smooth Artin stack with stable good moduli space
morphism $\pi\colon \ix \to \bX$. Then there exists a sequence
$\ix_n\to\ldots\to \ix_1\to\ix_0=\ix$ of birational morphisms of smooth Artin
stacks such that
\begin{enumerate}
\item 
Each $\ix_k$ admits a stable good moduli space
$\pi_k\colon \ix_k \to \bX_k$.
\item The morphism $\ix_{k+1} \to \ix_{k}$ is either a saturated blowup
  in a smooth center, or a root stack in a smooth divisor.
\item The morphism $\ix_{k+1} \to \ix_{k}$ induces a 
projective 
birational morphism of good moduli spaces $\bX_{k+1} \to \bX_{k}$.
\item $\bX_n$ is a smooth algebraic space.
\item $\ix_n\to \bX_n$ is a composition of a gerbe $\ix_n\to (\ix_n)_\rig$
and a root stack $(\ix_n)_\rig\to \bX_n$ in an snc divisor $D\subset \bX_n$.
\end{enumerate}
Moreover, the sequence is functorial with respect to strong smooth morphisms
$\ix'\to \ix$, that is, if $\bX'\to \bX$ is smooth and $\ix'=\ix\times_\bX
\bX'$, then the sequence $\ix'_n\to\ldots\to \ix'$ is obtained as the pullback
of $\ix_n\to\ldots\to \ix$ along $\ix'\to \ix$.
Similarly, the sequence is functorial with respect to arbitrary base change
$S'\to S$.
\end{corollary}
\begin{proof}
We first apply Theorem~\ref{thm.main} to $\ix$ and can thus assume that $\ix$
is a gerbe over a tame stack $\ix_\tame$.  We then apply destackification to
$\iy:=\ix_\tame$. This gives a sequence of smooth stacky blowups $\iy_n\to
\iy_{n-1}\to\ldots\to \iy_1\to \iy_0=\iy$, such that $\bY_n$ is smooth and
$\iy_n\to \bY_n$ factors as a gerbe $\iy_n\to (\iy_n)_\rig$ followed by a
root stack $(\iy_n)_\rig\to \bY_n$ in an snc
divisor. A smooth stacky blowup is either a root stack along a smooth divisor
or a blowup in a smooth center. A blowup on a tame stack is the same thing as
a saturated blowup.
We let $\ix_k=\ix\times_\iy \iy_k$. Then $\ix_n\to\iy_n\to \iy_\rig$ is a gerbe
and $\bX_n=\bY_n$.
\end{proof}

\begin{corollary}[Resolution of good quotient singularities]
If $X$ is a stable good moduli space of a smooth stack $\ix$, then there exists a
projective birational morphism $p\colon X'\to X$ where $X'$ is a smooth
algebraic space. The resolution is functorial with respect to smooth
morphisms and arbitrary base change but depends on $\ix$.
\end{corollary}

Considering Remark~\ref{rem.algorithm.notstable}, the result remains valid even
if $\ix\to X$ is not stable.

%%%%%%%%%%%%%%%%%%%%%%%%%%%%%%%%%%%%%%%%%%%%%%%%%%%%%%%%%%%%%%%%%%%%%%%%%%%%%

\section{The proof of Theorem \ref{thm.main} in the singular case}\label{sec.proof.sing}

Recall that the set of points $|\ix^\maxlocus|\subset |\ix|$ which has
maximal stabilizer dimension is closed. This set has a canonical structure
as a closed substack that we denote $\ix^\maxlocus$ (Appendix~\ref{app.fixed}).
If $\ix$ is smooth over a scheme $S$, then $\ix^\maxlocus$ is smooth over $S$.
If $f\colon \ix\to \iy$
is a stabilizer-preserving morphism, for example a closed immersion or a
strong morphism, then $\ix^\maxlocus=f^{-1}(\iy^\maxlocus)$ if $\ix$ and
$\iy$ have the same maximal stabilizer dimension. When $\ix=[U/G]$
and $G$ has the same dimension as the maximal stabilizer dimension, then
$\ix^\maxlocus=[U^{G_0}/G]$.

The locus of stable points $\ix^s$ may have connected components of different
stabilizer dimensions. A complication in the singular case is that the closures
of these components may intersect. The following lemma takes care of this problem. 

\begin{lemma} \label{lem.maxlocus-singular}
Let $\ix$ be an Artin stack with stable good moduli space morphism
$\pi\colon \ix \to X$.
Let $N$ (resp.\ $n$) be the maximum dimension of a stabilizer of a
point of $\ix$ (resp.\ a point of $\ix \smallsetminus \ix^s$). Let
$\ix^s_{(k)} \subset \ix^s$ denote the subset of points that are stable with
stabilizer of dimension $k$. Let $\ix_{\leq n}\subset \ix$ denote the
subset of points with stabilizer of dimension at most $n$.
Then
\begin{enumerate}
\item\label{item.singmaxlocus.1}
 $\ix^s$ is the disjoint union of the $\ix^s_{(k)}$.
\item\label{item.singmaxlocus.2}
  We have a partition of $\ix$ into open and closed substacks:
\begin{equation} \label{eq.stablestrat}
\ix = \ix_{\leq n} \amalg \ix^s_{(n+1)}\amalg\dots\amalg \ix^s_{(N)}.
\end{equation}
\end{enumerate}
In $\ix_{\leq n}$ we have the following two open substacks
\[
\ix^s_{(n)}\quad\text{and}\quad \ix^\ast=\ix_{\leq n}\smallsetminus \overline{\ix^s_{(n)}}.
\]
We let $\overline{\ix^s_{(n)}}$ and $\overline{\ix^\ast}$ denote their schematic
closures.
\begin{enumerate}\setcounter{enumi}{2}
\item\label{item.singmaxlocus.3}
  Every point of $\overline{\ix^s_{(n)}}$ has stabilizer of
  dimension $n$ and
  every point of $\overline{\ix^\ast}$ has stabilizer of dimension at most $n$.
\item\label{item.singmaxlocus.5}
  $\ix^s_{(n)}=\ix_{\leq n}\smallsetminus\overline{\ix^\ast}$. Thus, every
  point of $\ix\smallsetminus \ix^s$ with stabilizer of dimension $n$ is
  contained in $\overline{\ix^\ast}$. In particular,
  $\bigl\lvert \overline{\ix^\ast}^\maxlocus\bigr\rvert=\left\lvert \ix\smallsetminus \ix^s\right\rvert^\maxlocus$.
\end{enumerate}
\end{lemma}
\begin{proof}
\itemref{item.singmaxlocus.1}
The stabilizer dimension is locally constant on $\ix^s$, so
$\ix^s=\coprod_k \ix^s_{(k)}$.

\itemref{item.singmaxlocus.2}
The subset of points of stabilizer dimension $\geq n+1$ is closed by upper
semi-continuity. By assumption this set is also contained in $\ix^s$ and hence
open. This gives the decomposition in~\eqref{eq.stablestrat}.

\itemref{item.singmaxlocus.3}
Every point of either $\overline{\ix^s_{(n)}}$ or $\overline{\ix^\ast}$ lies in
$\ix_{\leq n}$ and thus has stabilizer of dimension at most $n$.
Every point of $\overline{\ix^s_{(n)}}$ has stabilizer of dimension at least $n$ by
upper semi-continuity.

\itemref{item.singmaxlocus.5}
Since $\ix^s$ is open, $\ix^s_{(n)}\subset \ix_{\leq n}\smallsetminus \overline{\ix^\ast}$.
Suppose that $x\in |\ix_{\leq n}|$. If $x$ has
stabilizer of dimension $<n$, then $x\in \ix^\ast$. If
$x$ has stabilizer of dimension $n$ but is not stable, then $x$ is the unique closed point
in $\pi^{-1}(\pi(x))$ and there exists a generization $y$ with stabilizer
of dimension $<n$ \cite[Proposition 9.1]{Alp:13}. Thus $y\in \ix^\ast$ and so $x\in \overline{\ix^\ast}$.
\end{proof}

\begin{remark}\label{rem.singular.simplercase}
  If all points of $\ix^s$ have stabilizers of the same dimension and
  $\ix\neq \ix^s$, then $\ix =\ix^\ast=
  \overline{\ix^\ast}$. Two notable examples are irreducible stacks and
  properly stable stacks. When this is the case, the proof in the singular case
  simplifies quite a bit. In general, note that the closed immersion
  $\overline{\ix^s_{(n)}}\cup \overline{\ix^\ast}\to \ix_{\leq n}$ is surjective
  but not necessarily an isomorphism if $\ix_{\leq n}$ has embedded components.
\end{remark}

\begin{example}
  Let $U = \Spec k[x,y,z]/(xz,yz)$ and let $\GG_m$ act with weights $(1,-1,0)$.
  Then $\ix = [U/\GG_m]$ has two irreducible components:
  ${\overline{\ix^s_{(1)}}=V(x,y)=\ix^\maxlocus}$ which has stabilizer $\GG_m$ at every
  point and $\overline{\ix^\ast}=V(z)=\overline{\ix^s_{(0)}}$
  which has a single point with stabilizer $\GG_m$. The algorithm will blow up
  $\overline{\ix^\ast}^\maxlocus=V(x,y,z)$.
\end{example}

\subsection{Proof of
Theorem~\ref{thm.main} when $\ix$ is singular and of finite type over an
algebraically closed field} \label{ssec.thm-main-algclosed} We use induction
on the maximum dimension of the stabilizer of a point
of $\ix \smallsetminus \ix^s$ and the smooth case to verify that the
maximum dimension drops after the appropriate saturated blowup.

First suppose $\ix=\ix^s$ and let us see that Theorem~\ref{thm.main}
holds with $\ix_n=\ix_0=\ix$. Note that $\ix$ has locally constant stabilizer
dimensions so it is a disjoint union of stacks with constant stabilizer
dimensions. If in addition $\ix$ is properly stable, then the stabilizer
dimensions are all equal to zero and $\ix$ is a tame stack
(Proposition~\ref{prop.tameisgood}). If instead
$\ix\smallsetminus \ie$ is a gerbe over a tame stack (automatic if
$\ix\smallsetminus \ie$ is reduced by Proposition \ref{prop.tamegerbe}), then
$(\ix\smallsetminus \ie)^\maxlocus=\ix\smallsetminus \ie$ on every connected
component of $\ix$ (Corollary~\ref{cor.tamegerbe-singular}). Since
$\ix\smallsetminus \ie$ is schematically dense in $\ix$ it follows that
$\ix^\maxlocus=\ix$ on every connected component of $\ix$.
Thus $\ix$ is a gerbe over a tame stack
(Corollary~\ref{cor.tamegerbe-singular}). We have now proven
Theorem~\ref{thm.main} when $\ix=\ix^s$.

Now suppose $\ix\neq \ix^s$. Let $n$ be the maximal dimension of a stabilizer
of $\ix\smallsetminus \ix^s$ and assume that the theorem has been proven for
smaller $n$. Let $\ix_0=\ix$ and $\ic_0=\overline{\ix^\ast}^\maxlocus$ in the
notation of Lemma~\ref{lem.maxlocus-singular}. Let $f_1\colon
\ix_1=\Bl^\pi_{\ic_0} \ix_0$ be the saturated blowup and
$\ie_1=f_1^{-1}(\ic_0)+ f_1^{-1}(\ie_0)$ where $f_1^{-1}(\ic_0)$ is the
exceptional divisor. That the conclusion of the main theorem holds for $\ell=0$
follows by the following proposition. In particular, the maximal dimension of a
stabilizer of $\ix_1\smallsetminus \ix_1^s$ is strictly less than $n$ so by
induction we have $\ix_n\to \ldots \to \ix_1$ such that the conclusions also
hold for $\ell=1,\dots,n$ and the theorem follows for $\ix$.

\begin{proposition} \label{prop.inductionstep-singular}
Let $\ix$ be an Artin stack of finite type over an algebraically closed field
with stable good moduli space morphism
$\pi\colon \ix \to \bX$.
Let $n$ be the maximum dimension of a stabilizer of a
point of $\ix \smallsetminus \ix^s$. Let 
$\ic=(\overline{\ix^\ast})^\maxlocus$
have the algebraic substack structure given by
Proposition~\ref{prop.ixmax-singular} and let $\ix' = \Bl^{\pi}_\ic(\ix)$ be
the saturated blowup. Then $\ix'$ is an Artin stack 
with the following properties.
\begin{enumerate}
\myitemi{2a}\label{item.satblowup.stablegms}
The stack $\ix'$ has a good moduli space $\bX'$ and the good moduli space
morphism
$\pi' \colon \ix' \to \bX'$ is stable (properly stable if $\pi$ is properly stable).
\myitemi{2b}\label{item.satblowup.isolocus}
  The morphism $f\colon \ix' \to \ix$ induces an isomorphism
  $\ix'\smallsetminus f^{-1}(\ic)\to \ix\smallsetminus
  \pi^{-1}\bigl(\pi(\ic)\bigr)$.
\myitemi{2c}\label{item.satblowup.gms-isolocus}
The induced morphism of good moduli spaces $\bX' \to \bX$
is 
projective
and an isomorphism over the image of $\ix^s$ in $\bX$.
\myitemi{3}\label{item.satblowup.stabdim-drops}
 Every point of $\ix' \smallsetminus (\ix')^s$
has stabilizer of dimension strictly less than $n$. 
\end{enumerate}
\end{proposition}
\begin{proof}
Assertions
\itemref{item.satblowup.stablegms}--\itemref{item.satblowup.gms-isolocus}
follow from the properties of the saturated blowup (Propositions
\ref{prop.saturatedblowups} and \ref{prop.stability.saturatedproj}).

Since $\ix^s_{(n)}\amalg \ix^\ast$ is open in $\ix_{\leq n}$ and its complement
$\ix_{\leq n}\smallsetminus (\ix^s_{(n)}\amalg \ix^\ast)$ is contained in
$\ic$, we conclude that $f^{-1}(\ix^s_{(n)}\amalg \ix^\ast)$
is schematically dense in $\Bl_\ic(\ix_{\leq n})$ and thus
\[
\Bl_\ic(\ix_{\leq n}) = \Bl_\ic(\overline{\ix^s_{(n)}}) \cup \Bl_\ic(\overline{\ix^\ast})
\]
and similarly for saturated blowups (Proposition~\ref{prop.saturatedblowups-and-basechange}\itemref{item.saturatedblowup.closed}). We note that every point of
$\Bl^\pi_\ic(\overline{\ix^s_{(n)}})$ has stabilizer of dimension $n$ by upper
semi-continuity. If every point of $\Bl^\pi_\ic(\overline{\ix^\ast})$ has stabilizer
of dimension $<n$, then the two components are necessarily disjoint and
every point of $\Bl^\pi_\ic(\overline{\ix^s_{(n)}})$ is stable.
To prove \itemref{item.satblowup.stabdim-drops} we may thus replace $\ix$ with
$\overline{\ix^\ast}$. Then $\ic=\ix^\maxlocus$ and the stabilizer of every stable point
has dimension strictly less than $n$.

By the local structure theorem \cite[Theorem 1.1]{AHR:15} we may assume that
$\ix=[U/G]$ where $G$ is the stabilizer of a point in $\ix^\maxlocus$
and $U=\Spec A$
is affine. By the local finiteness of group actions, there is a
finite-dimensional $G$-submodule $V\subset A$ such that $\Sym(V)\to A$ is
surjective. We thus have a closed $G$-equivariant embedding $U\to
V^\vee$. Consequently, we have a closed embedding of stacks $\ix=[U/G]\to
\iy=[V^\vee/G]$. Since $\ix^\maxlocus = \iy^\maxlocus \cap \ix$ we obtain a closed
embedding
\[
\Bl_{\ix^\maxlocus} \ix \to \Bl_{\iy^\maxlocus} \iy.
\]
Indeed, $\Bl_{\ix^\maxlocus} \ix$ is the strict transform, that is,
the closure of $\ix\smallsetminus \ix^\maxlocus$ in $\Bl_{\iy^\maxlocus} \iy$.
This also holds for saturated blowups by
Proposition~\ref{prop.saturatedblowups-and-basechange}.

From the smooth case (Proposition~\ref{prop.inductionstep}
\itemref{item.reichstein.stabdim-drops}), we know that
$\Bl^\pi_{\iy^\maxlocus} \iy$ has no points with stabilizer of dimension
$n$. Hence, neither does $\Bl^\pi_{\ix^\maxlocus} \ix$. This proves
\itemref{item.satblowup.stabdim-drops}.
\end{proof}
\subsection{General case}\label{ssec.sing.over.base}
Suppose now that $\ix$ is any noetherian stack with good moduli space $\bX$.
Then $\ix\to \bX$ is of finite type~\cite[Theorem A.1]{AHR:15}. Let $S=\bX$.
As in Lemma \ref{lem.maxlocus-singular}, let $n$ denote the maximum dimension
of a stabilizer of a point of $\ix \smallsetminus \ix^s$, let
$\ix^\ast=\ix_{\leq n}\smallsetminus \overline{\ix^s_{(n)}}$ and let
$\overline{\ix^\ast}$ denote its schematic closure.  Let $\ix'$ be the
saturated blowup of $\ix$ in $\overline{\ix^\ast}^\maxlocus$.

The only part of Section \ref{ssec.thm-main-algclosed} that needs to be modified is the
proof of
Proposition~\ref{prop.inductionstep-singular}\itemref{item.satblowup.stabdim-drops}:
that every point of $\ix' \smallsetminus (\ix')^s$ has stabilizer of dimension
strictly less than $n$. As in that proof, we may first reduce to the case where
$\ix=\overline{\ix^\ast}$.

By the local structure theorem (Theorem~\ref{thm.localstructuregeneral}) we may
assume that $\ix=[U/\GL_N]$, where $U=\Spec A$ is affine, and that $S=\bX=\Spec
R$ where $R=A^{\GL_N}$. Pick a finite-dimensional $\GL_N$-subrepresentation
$V\subseteq A$ such that $\Sym (V)\otimes_{\ZZ} R\to A$ is surjective. This
gives a closed embedding of stacks $\ix\to \iy$ where
$\iy=[\Spec(\Sym(V)\otimes_{\ZZ}R)/\GL_N]$ is smooth over $S$.

A subtlety in positive and mixed characteristic is that $\iy$ need not have a
good moduli space. We do have an \emph{adequate moduli space} $\iy\to \bY$ and
the induced map $\bX\to \bY$ is finite and injective. It is enough to show that
the maximum stabilizer dimension drops after replacing $S=\bX$ with the
localization at any point $x\in \bX$. Then $\iy\to \bY$ becomes a good moduli
space because the unique closed point of $\iy$ lies in $\ix$ and has linearly
reductive stabilizer~\cite[Theorem 9.3]{AHR:19}. Since $\iy$ is smooth over
$S$, we have that the maximum stabilizer dimension of $\Bl^\pi_{\iy^\maxlocus} \iy$
drops and since we have a closed embedding $\ix'=\Bl^\pi_{\ix^\maxlocus} \ix\to
\Bl^\pi_{\iy^\maxlocus} \iy$, the result follows.

\subsection{Functoriality for strong morphisms of singular stacks}
In the singular case, we have the following weaker version of
Theorem \ref{thm.reichsteinfunctorial}.
Let $f\colon \iy\to \ix$ be a strong morphism of stacks with stable good moduli
space morphisms $\pi_\ix \colon \ix \to \bX$, $\pi_\iy \colon
\iy \to \bY$. Assume that every point of
both $\ix^s$ and $\iy^s$ has
stabilizer of a fixed dimension $n$; e.g., that $\ix$ and $\iy$ 
are irreducible or properly
stable. Let $\ix'\to \ix$ and $\iy'\to \iy$ be the canonical morphisms produced
by Theorem~\ref{thm.main}.
\begin{proposition}[Functoriality for strong morphisms]
  Under the assumptions above, $\iy'$ is the schematic closure
  of $\iy^s$ in $\iy\times_\ix \ix'$.
  \end{proposition}

\begin{proof}
  Under our assumptions $\ix'$ and $\iy'$
  are produced by taking repeated saturated blowups in $\ix^\maxlocus$ and
  $\iy^\maxlocus$ respectively (Remark \ref{rem.singular.simplercase}).
  As in the proof of
  Theorem~\ref{thm.reichsteinfunctorial}, either $f^{-1}(\ix^\maxlocus)=\emptyset$
  which is trivial, or $f^{-1}(\ix^\maxlocus)=\iy^\maxlocus$.
  In the latter case, Proposition~\ref{prop.saturatedblowups-and-basechange} tells
  us that the saturated blowup
  $\Bl^{\pi_\iy}_{\iy^\maxlocus}\iy\to \iy$ is the strict transform of
  $\Bl^{\pi_\ix}_{\ix^\maxlocus}\ix\to \ix$ in the sense of Definition
  \ref{def.saturated-strict-transform}; that is, it is the schematic closure of
  $\iy\smallsetminus \iy^\maxlocus$ in $\iy\times_\ix
  \Bl^{\pi_\ix}_{\ix^\maxlocus}\ix$.
  After the final blowup, $\iy^s$ becomes
  schematically dense in $\iy'$ and the result follows.
\end{proof}

\begin{example}\label{ex.lci-sing}
  The following example shows that if $\iy$ is singular, then
  $\iy'$ need not equal the fiber product $\iy \times_\ix \ix'$ even if
  the morphism $\iy \to \ix$ is lci.
  
Let $\ix=[\A^3/\GG_m]$ where $\GG_m$ acts by weights $(1,1,-1)$. Then
$\ix^\maxlocus$ is the origin.  Let $\iy\subset \ix$ be the closed substack
defined by the ideal $(xz)$. Since $xz$ is invariant, $f\colon \iy\to \ix$ is a
strong regular embedding. The induced morphism of good moduli spaces is the
closed
immersion $\A^1\to \A^2$.

Since the maximum dimension of a stabilizer is one, the
canonical reduction of stabilizers is obtained by a single saturated blowup
along $\ix^\maxlocus$. Since $\ix$ is smooth, this is also the Reichstein transform.
The saturation of $\ix^\maxlocus$ is the substack defined by the ideal
$(xz,yz)$ which is the union of two irreducible substacks: the divisor $V(z)$ and
the codimension-two substack $V(x,y)$.
We conclude
that 
$\ix'=\bigl[\Spec
  k[\frac{x}{z},\frac{y}{z},z]\smallsetminus V(\frac{x}{z},\frac{y}{z})/\GG_m\bigr]$.
and thus
$\ix'\times_\ix \iy\subset \ix'$ is the closed substack defined by the
ideal $(z^2\frac{x}{z})$. On the other hand, the strict transform of $\iy$ is
$\iy'=\Bl^{\pi_\iy}_{\iy^\maxlocus} \iy$ which is cut out by $(\frac{x}{z})$.
\end{example}

\begin{remark}[Symmetric and intrinsic blowups]\label{rem.intrinsic-blowups}
Let $\ix$ be a stack with a properly stable good moduli space. Then our
algorithm repeatedly makes saturated blowups in
$\ix^\maxlocus$. An alternative would be to replace the saturated blowup by the
saturated \emph{symmetric blowup}: if $\ix^\maxlocus$ is defined by the
ideal sheaf $\iJ_\ix$, we may take the saturated Proj of the symmetric algebra
$\Sym(\iJ_\ix)$ instead of the saturated blowup which is the saturated Proj of
the
Rees algebra $\bigoplus_{k\geq 0} \iJ_\ix^k$. When $\ix$ is smooth, then so is $\ix^\maxlocus$ and the symmetric
algebra of $\iJ_\ix$ coincides with the Rees algebra of $\iJ_\ix$. If $\ix\to \iy$ is
a closed embedding into a smooth stack with a good moduli space, then
$\bigoplus_{k\geq 0} \iJ_\iy^k=\Sym(\iJ_\iy)\to \Sym(\iJ_\iy\otimes_{\cO_\iy}
\cO_\ix)\to \Sym(\iJ_\ix)\to \bigoplus_{k\geq 0} \iJ_\ix^k$ are surjective. We
thus have closed embeddings
\[
\Bl_{\ix^\maxlocus} \ix \to \SymBl_{\ix^\maxlocus} \ix \to (\Bl_{\iy^\maxlocus} \iy)\times_\iy \ix\to \Bl_{\iy^\maxlocus} \iy.
\]
In particular, the maximum stabilizer dimension of $\SymBl^\pi_{\ix^\maxlocus} \ix$
also drops.

For a quasi-projective scheme $U$ over $\CC$ with an action of a reductive
group $G$, Kiem, Li and Savvas~\cite{KiLi:13, KLS:17} have defined the
\emph{intrinsic blowup} $\Bl^G U$ of $U$. If $U\to V$ is a $G$-equivariant
embedding into a smooth scheme, then we have an induced embedding of stacks
$\ix=[U^{ss}/G]\to \iy=[V^{ss}/G]$ that admit good moduli spaces. The stack
$[\Bl^G U / G]$ is a closed substack of $(\Bl_{\iy^\maxlocus}
\iy)\times_\iy \ix$ that is slightly larger than both the blowup and the
symmetric blowup of $\ix$ in $\ix^\maxlocus$ and is independent on the choice
of a smooth embedding. It would be interesting to find a definition of the
intrinsic blowup for a stack with a good moduli space that (1) does not use
smooth embeddings and (2) does not use a presentation $\ix=[U^{ss}/G]$.
The intrinsic blowup also seems to be related to derived stacks
and blowups of such and it would be interesting to describe this relationship.
\end{remark}

%%%%%%%%%%%%%%%%%%%%%%%%%%%%%%%%%%%%%%%%%%%%%%%%%%%%%%%%%%%%%%%%%%%%%%%%%%%%%

\appendix
\section{The local structure theorem}
The simple version of the local structure theorem that we mainly use in this
paper is the following \cite[Theorem 4.12]{AHR:15}. Let $\ix$ be an algebraic
stack of finite type over an algebraically closed field $k$ with good moduli
space $\bX$. If $x\in |\ix|$ is a closed point with stabilizer $G_x$, then
there exists a cartesian diagram of stacks and good moduli spaces
\vspace{-3mm}% HACK
\[
\xymatrix{%
[U/G_x]\ar[r]\ar[d] & \ix\ar[d]^{\pi}\\
U\gitq G_x\ar[r] & \bX\ar@{}[ul]|\square
}
\]
where the horizontal maps are \'etale, $U$ is affine and there is a
closed point $u\in U$ above $x$. Note that $[U/G_x]\to \ix$ is stabilizer
preserving so $u$ is fixed by $G_x$.

When $k$ is merely perfect, it is not difficult to see that we can obtain such
a diagram after passing to a finite separable extension of $k$.  We will now
give a more general version from~\cite{AHR:19} that applies to any field $k$.
In fact, it does not even require $\ix$ to be of finite type over a field.

\begin{theorem}\label{thm.localstructuregeneral}
Let $\ix$ be a noetherian algebraic stack with good moduli space $\bX$. Let
$x\in |\ix|$ be a closed point. Then there exists
a group scheme $G\to \Spec \ZZ$, affine schemes $U$ and $V$ with actions of
$G$ and $\GL_n$,
a $G$-fixed closed point $u\in U$ and a commutative diagram
\vspace{-3mm}% HACK
\[
\xymatrix{%
[U/G]\ar[r]^-f\ar[d] & [V/\GL_n]\ar[r]^-g\ar[d] & \ix\ar[d]^{\pi}\\
U\gitq G\ar[r]    & V\gitq \GL_n\ar[r]^-h & \bX\ar@{}[ul]|\square
}
\]
where $f$ is finite \'etale, $g$ and $h$ are \'etale and $g(f(u))=x$.
Moreover,
\begin{enumerate}
\item $G_u$ is the connected component of the stabilizer group of $\Spec \kappa(u)\to U\to \ix$,
\item if $\characteristic(\kappa(x))=0$, then $G\to \Spec \ZZ$ is a Chevalley
  group, in particular smooth with connected fibers, and
\item if $\characteristic(\kappa(x))=p$, then $G\to \Spec \ZZ$ is
  diagonalizable.
\end{enumerate}
\end{theorem}
\begin{proof}
First note that $\bX$ is noetherian and $\pi$ is of finite type \cite[Theorem A.1]{AHR:15}. The result is now \cite[Corollary 13.4]{AHR:19}.
\end{proof}

Note that $U\gitq G\to V\gitq \GL_n$ is always finite but not necessarily
\'etale and that the left square is not cartesian unless $x$ has connected
stabilizer group.

%%%%%%%%%%%%%%%%%%%%%%%%%%%%%%%%%%%%%%%%%%%%%%%%%%%%%%%%%%%%%%%%%%%%%%%%%%%%%

\section{Gerbes and good moduli spaces}
Let $\pi\colon \ix\to \bX$ be a good moduli space morphism.  In this appendix,
we study when $\pi$ factors as $\ix\to \ix_\tame\to \bX$ where $\ix\to
\ix_\tame$ is a gerbe and $\ix_\tame\to \bX$ is a coarse moduli space. A
necessary condition is that the stabilizers of $\ix$ have locally constant
dimension. This is sufficient when $\ix$ is reduced
(Proposition~\ref{prop.tamegerbe}) and then $\ix_\tame$ is a tame stack.
In this appendix, we prove this when $\ix$ is of finite type over an algebraically
closed field.
The general case follows from Corollary~\ref{cor.tamegerbe-singular} where we
also give a precise condition when $\ix$ is not reduced.
When there is a factorization $\ix\to \ix_\tame\to \bX$ as above, then $\ix\to
\bX$ is a homeomorphism and in fact a coarse moduli space in the sense
of~\cite[Definition 6.8]{Rydh:13}.

\begin{proposition} \label{prop.tameisgood}
Let $\pi \colon \ix \to \bX$ be the good moduli
space of a stack such that all stabilizers are 0-dimensional. 
Then $\ix$ is a tame stack and $\bX$ is also the coarse
moduli space of $\ix$. Moreover, $\ix$ is separated if and only $\bX$ is separated.
\end{proposition}

\begin{proof}
By assumption, $\ix$ has quasi-finite and separated diagonal (recall that our
stacks have affine, hence separated, diagonals). Since $\ix$ has a good moduli
space, it follows that $\ix$ has finite inertia~\cite[Theorem 8.3.2]{Alp:14},
that is, $\ix$ is tame and $\bX$ is its coarse moduli space. Moreover, $\pi$ is
a proper universal homeomorphism, so $\ix$ is separated if and only if $\bX$ is
separated~\cite[Theorem 1.1(2)]{Con:05}.
\end{proof}
We will now generalize the previous proposition to stacks with constant dimensional stabilizers.
\begin{proposition} \label{prop.tamegerbe}
Let $\ix$ be a reduced
Artin stack with good moduli space $\pi \colon \ix \to \bX$.
If the dimension of the stabilizers of points of $\ix$ is constant,
then $\ix$ is a smooth gerbe over a tame stack $\ix_\tame$ whose coarse space
is $\bX$. In particular, if $\ix$ is smooth, then $\ix_\tame$ is smooth
and $\bX$ has tame quotient singularities.
\end{proposition}
\subsection{Reduced identity components}\label{ss.reduced.identity}
We begin with some preliminary results on reduced identity components
of group schemes in positive characteristic.

Let $G$ be an algebraic group of dimension $n$ over a perfect field $k$.
By \cite[Expos\'e VIa, Proposition 2.3.1]{SGA3}
or \cite[Tag \spref{0B7R}]{stacks-project} the identity
component $G^0$ of $G$ is an open and closed characteristic subgroup. Let
$G_0 = (G^0)_\red$ (non-standard notation). Since the field is perfect,
$G_0$ is a closed, smooth, subgroup scheme of $G^0$
\cite[Expos\'e VIa, 0.2]{SGA3} or \cite[Tag \spref{047R}]{stacks-project}.
Moreover, $\dim G_0 = \dim G = n$.

\begin{remark}\label{rem.normality}
In general, $G_0$ is not normal in $G^0$, for example, take $G=\GG_m\ltimes
\Galpha_p$. But if $G^0$ is diagonalizable, then $G_0\subset G^0$ is
characteristic, hence $G_0\subset G$ is normal.
Indeed, this follows from Cartier duality, since
the torsion subgroup of an abelian group is a characteristic subgroup.
\end{remark}

\begin{lemma} \label{lem.redidentity}
Let $S$ be a scheme and let $G\to S$ be a group scheme of finite type such that
$s\mapsto \dim G_s$ is locally constant. Let $H\subset G$ be a subgroup
scheme such that $H_{\overline{s}}=G_{\overline{s},0}$ for every geometric
point $\overline{s}\colon \Spec K\to S$. If $S$ is reduced, then there is at
most one such $H$ and $H\to S$ is smooth.
\end{lemma}
\begin{proof}
If $S$ is reduced, then $H\to S$ is smooth~\cite[Expos\'e VIb, Corollaire
  4.4]{SGA3}. If $H_1$ and $H_2$ are two different subgroup schemes as in the lemma,
then so is $H_1\cap H_2$. In particular, $H_1\cap H_2$ is also smooth.  By the
fiberwise criterion of flatness, it follows that $H_1\cap H_2=H_1=H_2$.
\end{proof}

Note that the lemma is also valid if $S$ is a reduced algebraic stack by
passing to a smooth presentation.

\begin{definition}
If $S$ is reduced and there exists a subgroup scheme $H\subset G$ as in the lemma,
then we say that $H$ is the \emph{reduced identity component} of $G$ and denote
it by~$G_0$.
\end{definition}

\begin{proposition}\label{prop.redidenty.inertia}
Let $\ix$ be a reduced algebraic stack 
such that every stabilizer has dimension
$d$. If $\ix$ admits a good moduli space,
then there exists a unique normal closed subgroup stack $(I \ix)_0\subset I\ix$
such that $(I\ix)_0\to \ix$ is smooth with connected fibers of dimension $d$.
Moreover, $I\ix/(I\ix)_0\to \ix$ is finite.
\end{proposition}
\begin{proof}[Proof of Proposition \ref{prop.redidenty.inertia} when
    $\ix$ is of finite type over an algebraically closed field $k$]
  \hfill

  \noindent
  By the local structure theorem of \cite[Theorem 4.12]{AHR:15}, for any closed
  point $x \in \ix(k)$ there is an affine scheme $U = \Spec A$
  and a cartesian diagram of stacks and good moduli spaces
  \vspace{-1mm}% HACK
  \[
  \xymatrix{
    [\Spec A/G] \ar[r]\ar[d] & \ix \ar[d]^{\pi}\\
    \Spec(A^G) \ar[r] & \bX
  }
  \]
  where the horizontal maps are \'etale neighborhoods of $x$ and $\pi(x)$
  respectively and $G=G_x$ is the stabilizer at $x$.
  Since, the diagram is cartesian, the map
  $[U/G] \to \ix$ is stabilizer preserving
  so the diagram of inertia groups
  \[
  \xymatrix{
    I([U/G]) \ar[r] \ar[d] & I\ix \ar[d]\\
    [U/G] \ar[r] & \ix
  }
  \]
  is also cartesian.
  Since $[U/G]$ is a quotient stack $I([U/G]) = [I_{G} U/G]$
  where $I_{G}U = \{(g,u) : gu = u\} \subset G  \times U$
  is the relative inertia group for the action of $G$ on $U$.
  Here $G$ acts on $I_{G}U$ via $h(g,u)=(hgh^{-1},hu)$.

  Note that $[G_0\times U/G]$, $I([U/G]) \subset
  [G\times U/G]$ are group schemes over $[U/G]$ with fibers
  of dimension $d$, and that $[G_0\times U/G]\to [U/G]$ is smooth (a twisted
  form of $G_0$).
  Lemma \ref{lem.redidentity} applied to
  $[G_0\times U/G]$, $I([U/G])$ and their intersection
  shows that $I([U/G])_0$
  exists and equals $[G_0\times U/G]$.

  If $\characteristic k = 0$ then $G_0 = G^0$ since $G$ is reduced and so $G_0$ is a normal
  subgroup of $G$. If $\characteristic k = p$ then, because $G$ is linearly reductive,
  $G^0$ is diagonalizable by Nagata's theorem~\cite[IV, \S 3, Theorem
    3.6]{DeGa:70}. Thus $G_0\subset G$ is normal as well
  (Remark~\ref{rem.normality}). Hence $I([U/G])_0\subset
  I([U/G])$ is a normal subgroup stack in all characteristics.

  Since $(-)_0$ is unique and commutes with \'etale base change, it follows by
  descent that $(I \ix)_0$ exists and is a normal closed subgroup stack.

  Finally, we note that $I\ix/(I \ix)_0$ is finite since
  $I_{G}U/(G_0\times U)\subset (G/G_0)\times U$ is a closed
  subgroup scheme of a finite group scheme.
\end{proof}

Note that in the proof we worked with the stack $[U/G]$ rather
than with the scheme $U$. The reason is that although $[U/G]$ is reduced
$U$ need not be reduced. However, if $\ix$ is smooth,
then one can arrange for $U$ to be smooth (Remark~\ref{rem.smoothness}).

\begin{proof}[Proof of Proposition \ref{prop.tamegerbe} when $\ix$ is of finite
  type over an algebraically closed field $k$]
  \hfill

  \noindent
  We have seen that the inertia stack $I\ix \to \ix$ contains a
  closed, normal subgroup stack $(I\ix)_0$ which is smooth over $\ix$, such
  that $I\ix/(I\ix)_0 \to \ix$ is finite with fibers that are linearly
  reductive finite groups (Proposition~\ref{prop.redidenty.inertia}).
  By \cite[Appendix A]{AOV:08}, $\ix$ is a
  gerbe over a stack $\ix\thickslash (I\ix)_0$ which is the
  rigidification of $\ix$ obtained by removing $(I\ix)_0$ from the
  inertia. The stack $\ix_\tame = \ix\thickslash (I\ix)_0$ will be the
  desired tame stack. In the \'etale chart in the proof of
  Proposition~\ref{prop.redidenty.inertia}, we have that
  $\ix_\tame = [U / (G/G_0)]$.

  The inertia of $\ix_\tame$ is finite and linearly reductive because its
  pullback to $\ix$ coincides with $I\ix/(I\ix)_0$ (or use the local
  description). Moreover, $\ix\to \ix_\tame$ has the universal property that a
  morphism $\ix\to \iy$ factors (uniquely) through $\ix_\tame$ if and only if
  $(I\ix)_0\to I\iy$ factors via the unit section $\iy\to I\iy$. In particular we
  obtain a factorization $\ix\to\ix_\tame\to \bX$ and $\ix_\tame\to \bX$ is
  the coarse moduli space since it is initial among maps to algebraic spaces.
\end{proof}
\begin{remark}
If $\ix$ is not reduced, then it need not be a gerbe over a tame stack. For
example, take $\ix=[\Spec k[x]/(x^n) / \GG_m]$ where $\GG_m$ acts by
multiplication.
\end{remark}
\begin{remark}
  In positive characteristic, if  $\ix$ is reduced and has constant stabilizer
  dimension but does not
  have a good moduli space then
there need not exist a subgroup stack $(I \ix)_0$. For a counter-example,
take $\ix=BG$ with $G=\GG_m\ltimes \Galpha_p$. Then $I(BG)$ is a reduced
algebraic stack. Also,
even if there is an open and closed subgroup stack $(I\ix)^0$ with connected fibers,
this subgroup stack need not be flat. For a counter-example, take $\ix=[\Spec k[x] /
  \Gmu_p]$ where $\Gmu_p$ acts with weight $1$. Then $(I\ix)^0=I\ix$ has connected fibers
but is not flat, and $(I\ix)_0$ is trivial.
\end{remark}
%%%%%%%%%%%%%%%%%%%%%%%%%%%%%%%%%%%%%%%%%%%%%%%%%%%%%%%%%%%%%%%%%%%%%%%%%%%%%
\section{Fixed loci of Artin stacks} \label{app.fixed}
The purpose of this section is to prove that when $\ix$ is an algebraic stack with a good moduli space, the locus of points with maximal dimensional stabilizer has a canonical closed substack structure.
\begin{lemma} \label{lem.maxlocus}
Let $\ix$ be an algebraic stack. Then the locus of points with
maximal-dimensional stabilizer is a closed subset of $|\ix|$.
\end{lemma}
\begin{proof}
  Since the representable morphism $I\ix \to \ix$ makes $I\ix$ into an
  $\ix$-group, the dimension of the fibers of the morphism is an upper
  semi-continuous function on $|\ix|$. Thus the locus of points with
  maximal-dimensional stabilizer is closed.
\end{proof}
\subsection{Stacks smooth over an algebraically closed field}
Let $X$ be a scheme or an algebraic space of finite type over a field $k$.
If a group scheme $G/k$ acts on $X$, then the fixed locus $X^G$ can be given a
canonical algebraic structure: $X^G$ represents the functor of $G$-equivariant
maps $T\to X$ where $T$ is equipped with the trivial action. The functor $X^G$
is represented by a closed subscheme~\cite[Proposition A.8.10 (1)]{CGP:10}
because the coordinate ring of $G$ is a projective (even free)
$k$-module. 
Moreover, if $G$ is linearly reductive and $X$ is smooth, then $X^G$ is also smooth~\cite[Proposition A.8.10 (2)]{CGP:10}.

\begin{proposition} \label{prop.ixmax}
  If $\ix$ is smooth over an algebraically closed field
  and admits a good moduli space,
  then the locus
$\ix^\maxlocus$ of points with maximal-dimensional stabilizer
  (with its reduced induced substack structure) is a closed smooth
substack. Moreover, if $\ie$ is an snc divisor, then $\ix^\maxlocus$ meets
$\ie$ with normal crossings.
\end{proposition}
\begin{proof}
Let $x$ be a closed point of $\ix^\maxlocus$ and let $G_x$ be its stabilizer group. 
By the local structure theorem \cite[Theorem 4.12]{AHR:15} (and Remark~\ref{rem.smoothness})
there is a smooth, affine scheme $U = \Spec A$ with an action of $G=G_x$ and a cartesian diagram of stacks and moduli spaces
\vspace{-2mm}% HACK
\[
\xymatrix{%
[U/G]\ar[r]\ar[d] & \ix\ar[d]^{\pi}\\
U\gitq G\ar[r] & \bX\ar@{}[ul]|\square
}
\]
where the horizontal arrows are \'etale and $u\in U$ is a fixed point above $x$.
It follows that $[U/G]^\maxlocus$ is the inverse image of $\ix^\maxlocus$
under an \'etale morphism. In particular $[U/G]^\maxlocus$ (with its reduced induced 
stack structure) is smooth at $u$ if and only if $\ix^\maxlocus$ is smooth at $x$.

As in \S\ref{ss.reduced.identity}, let $G_0$
be the reduced identity component of $G$.
Then $\dim G_0 = \dim G =n$
and any $n$-dimensional subgroup of $G$ necessarily contains $G_0$.

Since $G/G_0$ is finite, hence affine, so is $BG_0\to BG$. It follows that
$BG_0$ is cohomologically affine, that is, $G_0$ is linearly reductive.

A point of $U$ has maximal-dimensional stabilizer if and only
if it is fixed by the linearly reductive subgroup $G_0$. Thus
$[U/G]^\maxlocus = [U^{G_0}/G]$.
By \cite[Proposition A.8.10(2)]{CGP:10} $U^{G_0}$ is also smooth.
Note that $G$ acts on $U^{G_0}$ because 
$G_0$ is a characteristic, hence normal, subgroup of $G$ (Remark~\ref{rem.normality}).

If $\ie=\sum_{i=1}^m \ie_i$ is an snc divisor on $\ix$, let $D=\sum_{i=1}^m
D_i$ denote the pullback to $U$. For $J\subset \{1,2,\dots,m\}$, let
$D_J=\cap_{i\in J} D_i$. Note that $D$ is snc at $u$ since $D_J$ is smooth
at $u$ for every $J$ (Remark~\ref{rem.smoothness}). Since
$D_J^{G_0}=U^{G_0}\cap D_J$ is smooth at $u$ \cite[Proposition
  A.8.10(2)]{CGP:10}, it follows by flat descent that $\ix^\maxlocus\cap \ie_J$ is
smooth at $x$. Furthermore,
$T_{U^{G_0},x}=(T_{U,x})^{G_0}$ is the trivial part of the $G_0$-representation
$T_{U,x}$ and we obtain the equality
\vspace{-3mm}% HACK
\[
\begin{split}
\codim(U^{G_0}\cap D_J,U^{G_0})=\dim (N_{D_J/U,x})^{G_0}
  =\dim \Bigl( \bigoplus_{i\in J} N_{D_i/U,x}\Bigr)^{G_0} \\
=\sum_{i\in J} \dim (N_{D_i/U,x})^{G_0}
  =\sum_{i\in J} \codim(U^{G_0}\cap D_i,U^{G_0}).
\end{split}
\]
This shows that $\ix^\maxlocus$ meets $\ie$ with normal crossings.
\end{proof}

\begin{remark}
Proposition~\ref{prop.ixmax} holds more generally for any smooth algebraic stack
$\ix$ such that the stabilizer of every closed point has linearly reductive
identity component. Indeed, for any closed point $x\in \ix^\maxlocus$ there is
an \'etale
morphism $[U/G_x^0]\to \ix$ and $[U/G_x^0]^\maxlocus = [U^{G_0}/G_x^0]$ is smooth.
\end{remark}

\begin{remark}
For a fixed $n \geq 0$  the set $\ix^n$ of points whose stabilizer dimension is exactly $n$ is locally closed. However $\ix^n$ (with its reduced induced substack structure) need not be smooth. For example, if $\mathfrak{so}_5$ is the adjoint representation of $\SO_5$ and $\ix=[\mathfrak{so}_5/\SO_5]$, then $\ix^4$ has two irreducible components that intersect \cite[Beispiel 6.6(a)]{BoKr:79}.
In addition, the open substack of points whose stabilizer has dimension at most $n$ need not be saturated and consequently  may not have a good moduli space. For example let $\ix = 
[\left(\Proj (\Sym^4 \CC^2)\right)^{ss}/\PSL_2]$. The open substack of points with 0-dimensional stabilizer does not admit a coarse (and hence good) moduli space \cite[Remark 9.2] {KeMo:97}.
\end{remark}

\subsection{General case}
We will now consider noetherian stacks, possibly of mixed characteristic.
We will consider group schemes $G$ over $\ZZ$ that are flat and of finite type
and either
\begin{enumerate}
\item $G$ is diagonalizable, or
\item $G$ is a Chevalley group, that is, smooth with connected split reductive fibers.
\end{enumerate}
In either case we can make sense of the reduced connected component $G_0$. In
the first case, $G=D(A)$ and $G_0=D(A/A_{\mathrm{tor}})$ is a torus. In the second case
$G=G^0=G_0$. The coordinate ring of $G_0$ (as well as $G$) is a free
$\ZZ$-module. Given a scheme (or an algebraic space) $X\to S$ with an action of
$G$, the functor $X^{G_0}$ is therefore represented by a closed
subscheme of $X$~\cite[Proposition A.8.10 (1)]{CGP:10}.
If $(G_0)_S$ is linearly reductive and $X\to S$ is smooth, then $X^{G_0}\to
S$ is also smooth~\cite[Proposition A.8.10 (2)]{CGP:10}.

If $X$ is an arbitrary algebraic space and $\ix=[X/G]$ where $X$ has a $G$-fixed
point, then we can equip $\ix^\maxlocus$ with the substack structure of 
$[X^{G_0}/G]$. The next proposition shows that this substack structure
on $\ix^\maxlocus$ is independent on the presentation $\ix=[X/G]$. Combining
this fact with the local structure theorem
(Theorem~\ref{thm.localstructuregeneral}) we can conclude
that if $\ix$ is an arbitrary stack with a good moduli space, then $\ix^\maxlocus$
has a canonical substack structure. To
achieve this, we start with a slightly different definition of $\ix^\maxlocus$,
not referring to $G$.

\begin{proposition}\label{prop.ixmax-singular}
Let $\ix$ be a noetherian Artin stack. Let $n$ be
the maximal dimension of the stabilizer groups. Consider the functor
$F\colon (\mathrm{Sch}_{/\ix})^{\mathrm{op}}\to
\mathrm{(Set)}$ where $F(T\to \ix)$ is the set of closed subgroup schemes $H_0\subset
I\ix \times_{\ix} T$ that are smooth over $T$ with connected fibers of dimension $n$.
If $\ix$ admits a good moduli space,
then the functor is represented by a closed substack $\ix^\maxlocus$.
In
particular, for any $T\to \ix$, there is at most one such subgroup scheme $H_0$ and it is characteristic.
Moreover,
\begin{enumerate}
\item $|\ix^\maxlocus|$ is the set of points with stabilizer of dimension $n$.
\item If $\ix=[X/G]$ where $X$ is a separated algebraic space and
$G\to \Spec \ZZ$ is of relative dimension $n$ and as
above, then $\ix^\maxlocus=[X^{G_0}/G]$.
\item The following are equivalent (a) $\ix=\ix^\maxlocus$, (b) $I\ix$ contains
  a closed subgroup stack, smooth over $\ix$ with connected fibers of dimension~$n$,
  and (c) $I\ix$ contains a closed subgroup stack, flat over $\ix$
  with fibers of dimension $n$.
\item If $\ix$ is smooth over an algebraic space $S$, then $\ix^\maxlocus$ is
  smooth over $S$. If, in addition, $\ie$ is an snc divisor on $X$ relative
  $S$ (Definition~\ref{def.snc}), then
  $\ix^\maxlocus$ has normal crossings with $\ie$.
\end{enumerate}
\end{proposition}

Let $\ix'\to \ix$ be a stabilizer-preserving morphism.
It follows from the functorial description of $\ix^\maxlocus$ that
if $\ix'$ and $\ix$ have the same maximal
stabilizer dimension, then $\ix'^\maxlocus = \ix^\maxlocus\times_\ix \ix'$.
If not, then $\ix^\maxlocus\times_\ix \ix' = \emptyset$.

Similarly, if $\ix'\to \ix$
is \'etale, representable and separated, but not necessarily
stabilizer-preserving, then
$I\ix'\subset (I\ix)\times_\ix \ix'$ is open and closed and
$\ix^\maxlocus\times_\ix \ix'=\ix'^\maxlocus$
or $\ix^\maxlocus\times_\ix \ix'=\emptyset$.

\begin{proof*}[Proof of Proposition~\ref{prop.ixmax-singular}]
(1) is an immediate consequence of the main claim since if $T=\Spec k$ with $k$ perfect,
then $H_0=(I\ix\times_\ix T)_0$ is the unique choice of $H_0$.
(3a)$\implies$(3b) follows by definition and (3b)$\implies$(3c) is trivial.

To prove the main claim, (3c)$\implies$(3a) and (4) we may work \'etale-locally
around a closed point $x\in |\ix|$ with
stabilizer of dimension $n$. More precisely, if $\ix'\to \ix$ is \'etale,
representable and separated and $\ix'$ has a good moduli space, then
it is enough to deduce the result for $\ix'$.
Using the local structure theorem (Theorem~\ref{thm.localstructuregeneral}),
we can thus assume that $\ix=[X/G]$ where $G$ is either diagonalizable
or split reductive of relative dimension $n$ over $\Spec \ZZ$.
The main claim thus follows from (2).
The following lemma
applied to $H=I\ix\times_\ix T$ implies (2) and (3c)$\implies$(3a).
If $\ix=[X/G]\to S$ is smooth, then $\ix^\maxlocus=[X^{G_0}/G]\to S$ is
smooth~\cite[Proposition A.8.10 (2)]{CGP:10}\footnote{If $x$ has characteristic
  zero and $S$ has mixed characteristic, then $G\times_\ZZ S\to S$ is reductive
  but need not be linearly reductive so \cite[Proposition A.8.10 (2)]{CGP:10}
  does not apply.  We can, however, verify smoothness after replacing $\ix \to
  S$ with the base change along $S':=\Spec \cO^h_{S,s}\to S$ where
  $s$ is the image of $x$. Since $S'$ has characteristic zero,
  $G\times_\ZZ S'\to S'$ is linearly reductive and loc. cit. applies.}.
If in addition $\ie=\bigcup_{i\in I} \ie_i$ is an snc divisor, then
$\ix^\maxlocus$ meets $\ie$ with normal crossings since this can be verified
after replacing $S$ with a base change to an algebraically closed field
where we can conclude by Proposition~\ref{prop.ixmax}. We have now
proved (4).
\end{proof*}

\begin{lemma}
Let $G$ be a group scheme of dimension $n$ over $\Spec \ZZ$ and assume that
$G$ is either diagonalizable or a Chevalley group.
Let $T$ be a scheme and let $H\subset G\times T$ be a closed subgroup scheme.
Then the following are equivalent:
\begin{enumerate}
\item $G_0\times T \subset H$.
\item There exists a closed subgroup scheme $H_0\subset H$ that is smooth over $T$
  with connected fibers of dimension $n$.
\item There exists a closed subgroup scheme $H_1\subset H$ that is flat over $T$
  with fibers of dimension $n$.
\end{enumerate}
and under these conditions $H_0$ is unique and equals $G_0\times T$.

In particular, the functor $F$ representing closed subgroup schemes $H_0\subset H$ that
are smooth with connected fibers of dimension $n$ coincides with
the functor representing the condition that $G_0\times T\subset H$
and this is represented by a closed subscheme of $T$.
\end{lemma}
\begin{proof}
Clearly (1)$\implies$(2)$\implies$(3).
If $G\to \Spec \ZZ$ is smooth and connected of dimension $n$, then
any closed subgroup scheme $H_1\subset
G\times T$ that is flat with fibers of dimension $n$ necessarily equals
$G\times T$ by the fiberwise criterion of flatness. Thus (3)$\implies$(1).
If instead $G$ is diagonalizable, then so is $H_1$. This follows from
\cite[Expos\'e X, Corollaire 4.8]{SGA3} since $H_1$ is commutative
and ${}_n H_1$ is a closed subscheme of ${}_n G\times T$ which is finite over
$T$. It follows that $G_0\times T \subset H_1$ so (1) holds. The fiberwise
criterion of flatness also shows that if $H_0$ is as in (2), then
$G_0\times T\subset H_0$ is an equality.

Thus, the functor $F$ represents the condition that
$G_0\times T \subset H$, or equivalently,
the condition that $H\cap (G_0\times T) \subset G_0\times T$ is the identity, that is,
the Weil restriction
\[
F=\prod_{G_0\times T/T} (H\cap G_0\times T)/(G_0\times T).
\]
This functor is
represented by a closed subscheme since the coordinate ring of $G_0$ is a
free $\ZZ$-module, see~\cite[Expos\'e VIII, Th\'eor\`eme 6.4]{SGA3}.
\end{proof}

\begin{corollary}\label{cor.tamegerbe-singular}
Let $\ix$ be a connected stack with good moduli space $\bX$. Then $\ix$ is a gerbe
over a tame stack if and only if $\ix^\maxlocus =\ix$. More precisely:
\begin{enumerate}
\item\label{item.tamegerbe-singular}
  If $\ix=\ix^\maxlocus$, then there exists a unique closed subgroup stack
  $(I\ix)_0\subset I\ix$ that is smooth over $\ix$ with connected fibers
  of maximal dimension. Moreover, $I\ix / (I\ix)_0\to \ix$ is finite
  and $\ix$ is a smooth gerbe over the tame stack
  $\ix\thickslash (I\ix)_0$.
\item If $\ix$ is an fppf gerbe over a tame stack, then $\ix=\ix^\maxlocus$.
\end{enumerate}
\end{corollary}
\begin{proof}
If $\ix=\ix^\maxlocus$, then $(I\ix)_0$ exists by
Proposition~\ref{prop.ixmax-singular}. Consider the
rigidification $\ix_\tame=\ix \thickslash (I\ix)_0$. By construction, the
stabilizer groups of $\ix_\tame$ are finite and linearly reductive but we do
not yet know that the inertia $I\ix_\tame\to \ix_\tame$ is finite, or
equivalently, that $I\ix/ (I\ix)_0\to \ix$ is finite.

That the inertia is finite can be verified \'etale-locally on $\bX$. By the
local structure theorem (Theorem~\ref{thm.localstructuregeneral}) we can thus
assume that we have a finite \'etale surjective morphism $f\colon [U/G]\to \ix$
with $G\to \Spec \ZZ$ diagonalizable or smooth with connected fibers. Note that
$f^*(I\ix)_0=[G_0 \times_\ZZ U/G]$.  We thus obtain a cartesian diagram
\[
\xymatrix{
[U/G]\ar[r]^-f\ar[d] & \ix\ar[d]\\
[U/(G/G_0)]\ar[r] & \ix_{\mathrlap{\tame}}\ar@{}[ul]|\square
}
\]
where the vertical maps are smooth gerbes and the horizontal maps are
finite \'etale. Since $U$ is separated and $G/G_0$ is finite, it follows that
$\ix_\tame$ has finite diagonal.

Conversely, if $f\colon \ix\to \ix_\tame$ is an fppf-gerbe, then
$If\subset I\ix$ is a flat normal subgroup stack and $I\ix/If=f^*I\ix_\tame$ is
finite. Thus $If$ is closed with fibers of the same dimension as the
stabilizers of $\ix$ and it follows that $\ix^\maxlocus=\ix$ by
Proposition~\ref{prop.ixmax-singular}.
\end{proof}

If $\ix$ is a reduced stack such that every stabilizer has dimension $d$,
then $\ix^\maxlocus=\ix$ so Propositions~\ref{prop.tamegerbe}
and~\ref{prop.redidenty.inertia} are special cases of 
Corollary~\ref{cor.tamegerbe-singular} \itemref{item.tamegerbe-singular}.

\bibliographystyle{dary}
\bibliography{refs}
\end{document}